\documentclass[a4paper,reqno]{amsart}

\usepackage{graphicx}
\usepackage{mathptmx}
\usepackage[usenames,dvipsnames]{color}
\usepackage{microtype}
\usepackage{amsmath}
\usepackage{amssymb}
\usepackage{amsthm}
\usepackage{txfonts}
\usepackage{dsfont}
\usepackage{mathrsfs}
\usepackage{subfig}
\usepackage{enumitem}
\usepackage{ulem}
\usepackage[colorlinks = true, citecolor = black, linkcolor = black, urlcolor = black, pdfstartview=FitH]{hyperref}

\newtheorem{theorem}{Theorem}[section]
\newtheorem{corollary}{Corollary}

\newtheorem{lemma}[theorem]{Lemma}
\newtheorem{proposition}{Proposition}

\newtheorem{example}{Example}

\theoremstyle{definition}
\newtheorem{definition}[theorem]{Definition}
\newtheorem{remark}{Remark}

\title[Intermediate $\beta$-shifts of finite type]{Intermediate $\beta$-shifts of finite type}

\thanks{The first author was supported by NSFC (no.\ 11201155 and 11371148) and Guangdong Natural Science Foundation 2014A030313230. The second author was supported by the ERC grant no.\ 306494. The third author thanks the Universit\"at Bremen and the SCUT for their support.}

\author{Bing Li}
\address{Department of Mathematics, South China University of Technology, Guangzhou, 510641, P.R. China}
\email{scbingli@scut.edu.cn}

\author{Tuomas Sahlsten}
\address{Einstein Institute of Mathematics, The Hebrew University of Jerusalem, Givat Ram, Jerusalem 91904, Israel}
\email{tsahlsten@math.huji.ac.il}

\author{Tony Samuel}
\address{Fachbereich Mathematik, Universit\"at Bremen, 28359 Bremen, Germany}
\email{tony@math.uni-bremen.de}

\begin{document}

\maketitle

\begin{abstract}
An aim of this article is to highlight dynamical differences between the greedy, and hence the lazy, $\beta$-shift (transformation) and an intermediate $\beta$-shift (transformation), for a fixed $\beta \in (1, 2)$.  Specifically, a classification in terms of the kneading invariants of the linear maps $T_{\beta,\alpha} \colon x \mapsto \beta x + \alpha \bmod 1$ for which the corresponding intermediate $\beta$-shift is of finite type is given.  This characterisation is then employed  to construct a class of pairs $(\beta,\alpha)$ such that the intermediate $\beta$-shift associated with $T_{\beta, \alpha}$ is a subshift of finite type.  It is also proved that these maps $T_{\beta,\alpha}$ are not transitive.  This is in contrast to the situation for the corresponding greedy and lazy  $\beta$-shifts and $\beta$-transformations, for which both of the two properties do not hold.
\end{abstract}

\section{Introduction, motivation and statement of main results}

\subsection{Introduction and motivation}

\noindent For a given real number $\beta \in (1, 2)$ and a real number $x \in [0,1/(\beta-1)]$, an infinite word $(\omega_{n})_{n \in \mathbb{N}}$ over the alphabet $\{ 0, 1\}$ is called a $\beta$-\textit{expansion} of the point $x$ if
\begin{align*}
x = \sum_{k= 1}^{\infty} \omega_{k} \beta^{-k}.
\end{align*}
If $\beta$ is a natural number, then the $\beta$-expansions of a point $x$ correspond to the $\beta$-adic expansions of $x$.  In this case, almost all positive real numbers have a unique  $\beta$-expansion.  On the other hand, in \cite{S:2003b} it has been shown that if $\beta$ is not a natural number, then, for Lebesgue almost all $x$, the cardinality of the set of $\beta$-expansions of $x$ is equal to the cardinality of the continuum.

The theory of $\beta$-expansions originates with the works of R\'enyi \cite{R:1957} and Parry \cite{P:1960,P1964}, where an important link to symbolic dynamics has been established.  Indeed, through iterating the greedy $\beta$-transformation $G_{\beta} \colon x \mapsto \beta x \bmod 1$ and the lazy $\beta$-transformation $L_{\beta} \colon x \mapsto \beta (x - 1) +2 \bmod 1$ one obtains subsets of $\{0, 1\}^{\mathbb{N}}$ whose closures are known as the greedy and (normalised) lazy $\beta$-shifts, respectively.  Each point $\omega^{+}$ of the greedy $\beta$-shift is a $\beta$-expansion, and corresponds to a unique point in $[0, 1]$; and each point $\omega^{-}$ of the lazy $\beta$-shift is a $\beta$-expansion, and corresponds to a unique point in $[(2-\beta)/(\beta - 1), 1/(\beta - 1)]$.  (Note that in the case when $(2-\beta)/(\beta - 1) \leq 1$, if $\omega^{+}$ and $\omega^{-}$ are \mbox{$\beta$-expansions} of the same point, then $\omega^{+}$ and $\omega^{-}$ do not necessarily have to be equal, see \cite{KL:1998}.)   Through this connection we observe one of the most appealing features of the theory of $\beta$-expansions, namely that it links symbolic dynamics to number theory.  In particular, one can ask questions of the form, for what class of numbers the greedy and lazy $\beta$-shifts have given properties and vice versa.   In fact, although the arithmetical, Diophantine and ergodic properties of the greedy and lazy $\beta$-shifts have been extensively studied (see \cite{B:1989,DK:2002,S:2003} and references therein), there are many open problems of this form.  Further, applications of this theory to the efficiency of analog-to-digital conversion have been explored in \cite{DDGV06}.  Moreover, through understanding $\beta$-expansions of real numbers, advances have been made in understanding Bernoulli convolutions.  For literature in this direction, we refer the reader to \cite{DK:2002b,DV:2005,DV:2007} and reference therein.

There are many ways, other than using the greedy and lazy $\beta$-shifts, to generate a \mbox{$\beta$-expansion} of a positive real number. For instance the intermediate $\beta$-shifts $\Omega_{\beta, \alpha}$ which arise from the intermediate $\beta$-transformations $T_{\beta, \alpha}^{\pm} \colon [0, 1] \circlearrowleft$.  These transformations are defined for $(\beta,\alpha) \in \Delta$, with
\[
\Delta \coloneqq \{ (\beta, \alpha) \in \mathbb{R}^{2} \colon \beta \in (1, 2) \; \text{and} \; 0 \leq \alpha \leq 2 - \beta\},
\]
as follows (also see Figure~\ref{Fig1}).  Letting $p = p_{\beta, \alpha} \coloneqq (1-\alpha)/\beta$ we set
\[
T^{+}_{\beta, \alpha}(p) \coloneqq 0
\quad \text{and} \quad
T_{\beta,\alpha}^{+}(x) \coloneqq \beta x + \alpha \bmod 1,
\]
for all $x \in [0, 1] \setminus \{ p\}$.  Similarly, we define
\[
T^{-}_{\beta, \alpha}(p) \coloneqq 1
\quad \text{and} \quad
T_{\beta,\alpha}^{-}(x) \coloneqq \beta x + \alpha \bmod 1,
\]
for all $x \in [0, 1] \setminus \{ p\}$.� Indeed we have that the maps $T_{\beta, \alpha}^{\pm}$ are equal everywhere except at the point $p$ and that
\[
T^{-}_{\beta, \alpha} (x) = 1 - T_{\beta, 2 - \beta - \alpha}^+(1 - x),
\]
for all $x \in [0,1]$. Observe that, when $\alpha = 0$, the maps $G_{\beta}$ and $T^{+}_{\beta, \alpha}$ coincide, and when $\alpha = 2-\beta$, the maps $L_{\beta}$ and $T^{-}_{\beta, \alpha}$ coincide.� Here we make the observation that for all $(\beta, \alpha) \in \Delta$, the symbolic space $\Omega_{\beta, \alpha}$ associated to $T_{\beta, \alpha}^{\pm}$ is always a subshift, meaning that it is closed and invariant under the left shift map (see Corollary~\ref{cor:shift_invarance}).

The maps $T_{\beta, \alpha}^{\pm}$ are sometimes also called linear Lorenz maps and arise naturally from the Poincar\'e maps of the geometric model of Lorenz differential equations.  We refer the interested reader to \cite{EO:1994,L:1963,V:2003,W:1980} for further details.

\begin{figure}[htbp]
\centering
\subfloat[Plot of $L_{\beta}$.]{
\scalebox{0.325}{
\includegraphics{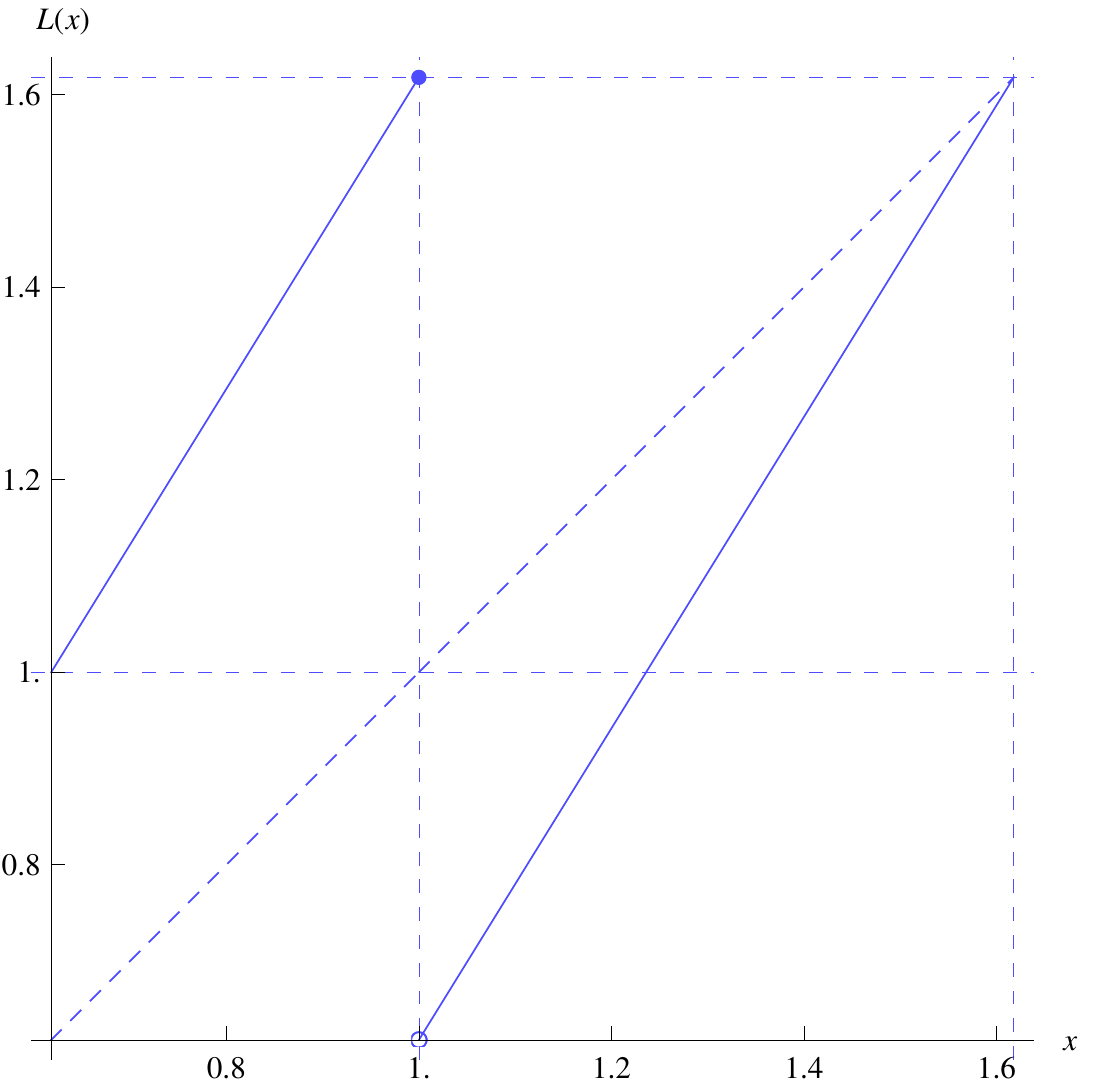}	
	}}
	\hspace{1em}
\subfloat[Plot of $T_{\beta,\alpha}^{+}$.]{
\scalebox{0.325}{
\includegraphics{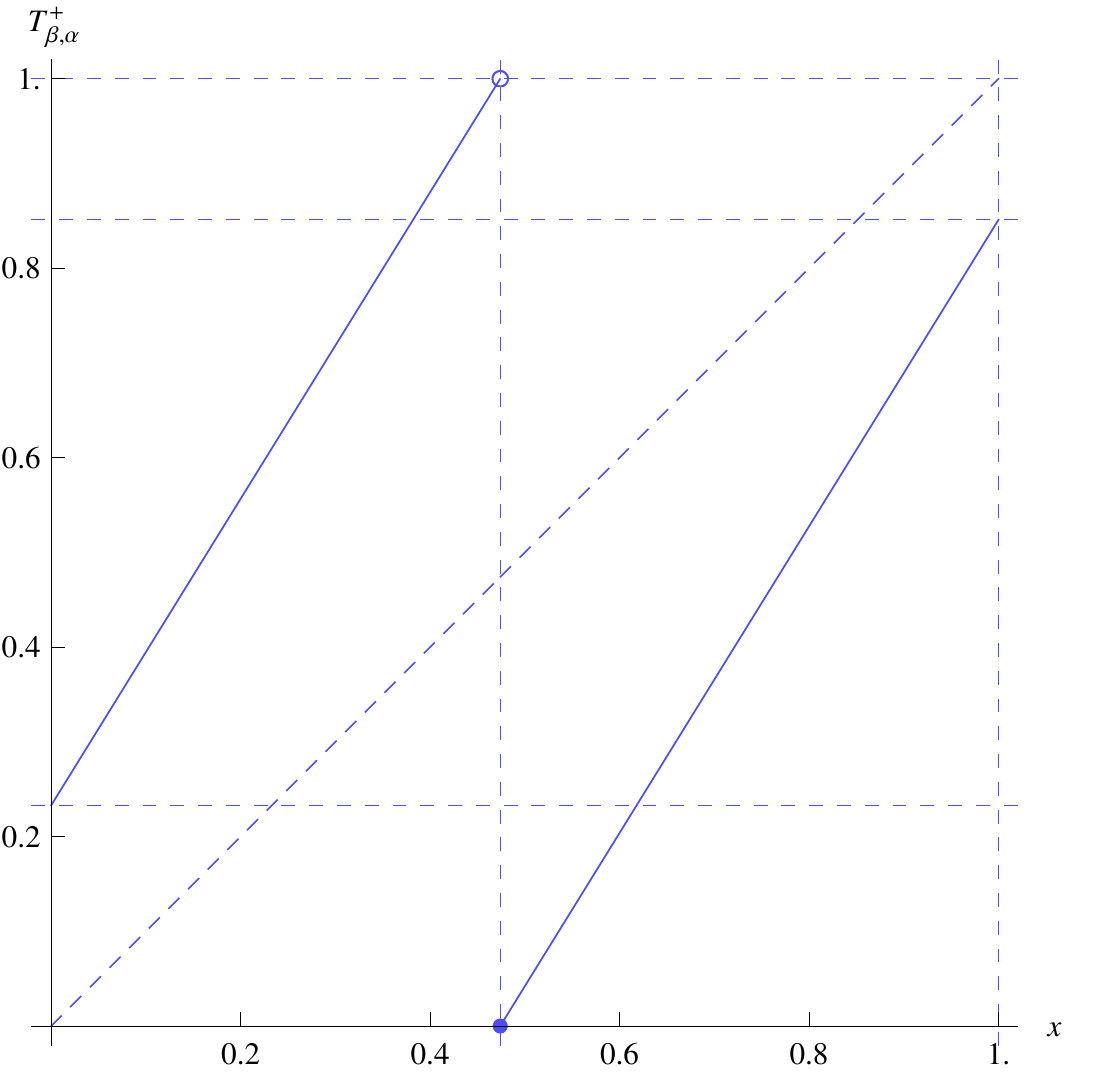}	
	}}
	\hspace{1em}
\subfloat[Plot of $G_{\beta}$.]{
\scalebox{0.325}{
\includegraphics{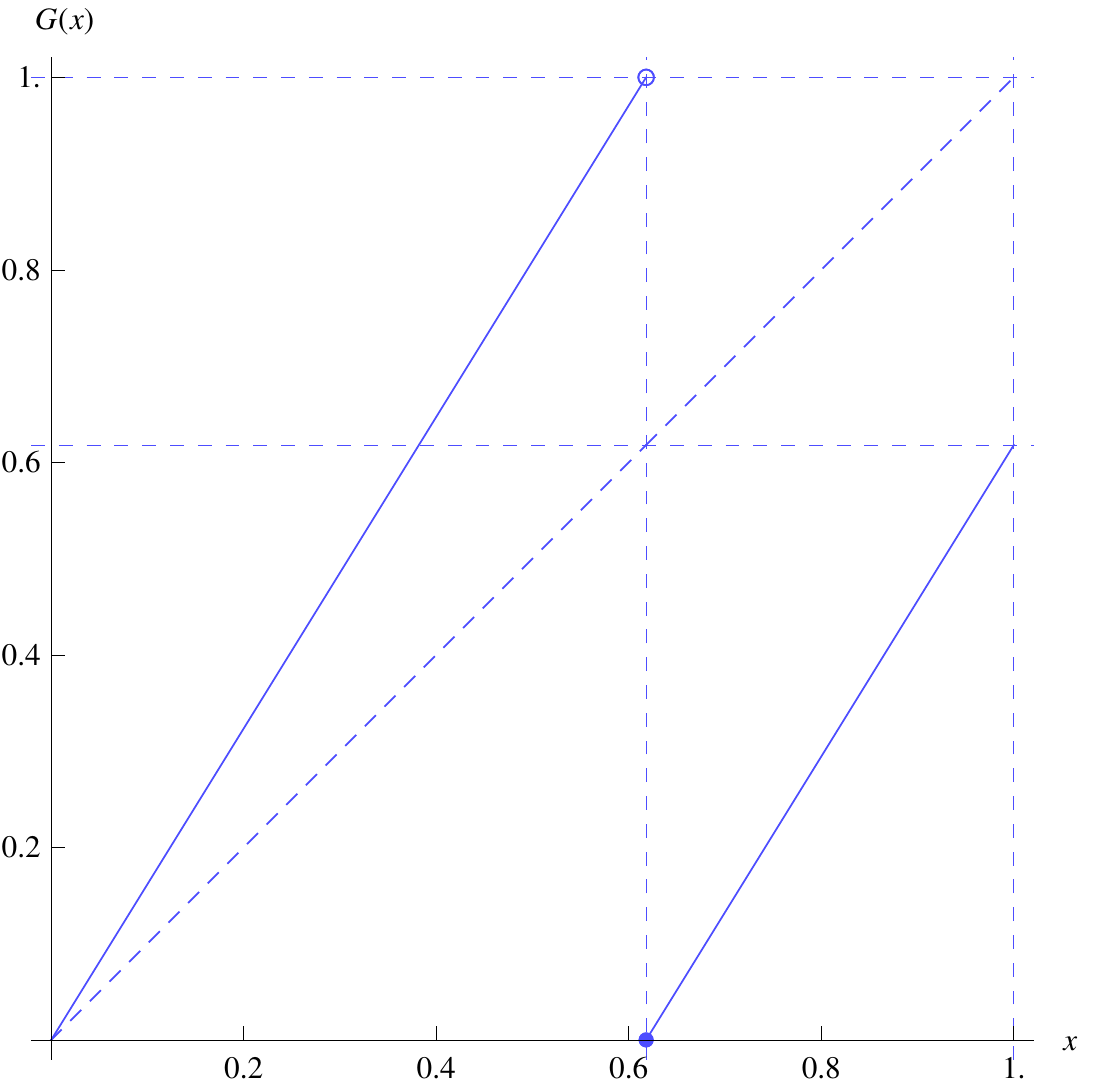}	
	}}
\caption{Plot of $T_{\beta,\alpha}^{+}$ for $\beta = (\sqrt{5} + 1)/2$ and $\alpha = 1 - 0.474 \beta$, and the corresponding lazy and greedy $\beta$-transformations. (The height of the filled in circle determines the value of the map at the point of discontinuity.)}
\label{Fig1}
\end{figure}

In this article the intermediate $\beta$-transformations and $\beta$-shifts are the main topic of study, and thus, to illustrate their importance we recall some of the known results in this area.  Parry \cite{P:1965} proved that any topological mixing interval map with a single discontinuity is topologically conjugate to a map on the form $T_{\beta, \alpha}^{\pm}$ where $(\beta, \alpha) \in \Delta$.  In \cite{G:1990,HS:1990} it is shown that a topologically expansive piecewise continuous map $T$ can be described up to topological conjugacy by the kneading invariants of the points of discontinuity of $T$.  (In the case that $T = T_{\beta, \alpha}^{\pm}$, the kneading invariants of the point of discontinuity $p$ are precisely given by the points in the associated intermediate $\beta$-shift which are $\beta$-expansion of $p + \alpha/(\beta - 1)$.)  For such maps, assuming that there exists a single discontinuity, the authors of \cite{BV:2012,HS:1990} gave a simple condition on pairs of infinite words in the alphabet $\{ 0, 1\}$ which is satisfied if and only if that pair of sequences are the kneading invariants of the point of discontinuity of $T$.

Our main results, Theorems~\ref{thm:SFTandTransitive} and \ref{thm:ESFTP} and Corollary~\ref{cor:fibersSFT}, contribute to the ongoing efforts in determining the dynamical properties of the intermediate $\beta$-shifts and examining whether these properties also hold for the counterpart greedy and lazy $\beta$-shifts.  We demonstrate that it is possible to construct pairs $(\beta,\alpha) \in \Delta$ such that the associated intermediate $\beta$-shift is a subshift of finite type  but for which the maps $T_{\beta, \alpha}^{\pm}$ are not transitive.  In contrast to this, the corresponding greedy and lazy $\beta$-shifts are not subshifts of finite type (neither sofic shifts, namely a factor of a subshift of finite type) and, moreover, the maps $G_{\beta}$ and $L_{\beta}$ are transitive.   Recall that an interval map $T \colon [0, 1] \circlearrowleft$ is called \textit{transitive} if and only if for all open subintervals $J$ of $[0, 1]$ there exists $m \in \mathbb{N}$ such that $\bigcup_{i = 1}^{m} T^{i}(J) = (0,1)$.

To prove the transitivity part of our result we apply a result of Palmer \cite{P:1979} and Glendinning \cite{G:1990}, where it is shown that for any $1<  \beta < 2$ the maps $G_{\beta}$ and $L_{\beta}$ are transitive.  In fact, they give a complete classification of the set of points $(\beta, \alpha) \in \Delta$ for which the maps $T_{\beta, \alpha}^{\pm}$ are not transitive.  (For completeness we restate their classification in Section~\ref{sec:TheoremB}.)

Before stating our main results let us emphasis the importance of subshifts of finite type.  These symbolic spaces give a simple representation of dynamical systems with finite Markov partitions.  There are many applications of subshifts of finite type, for instance in coding theory, transmission and storage of data or tilings.  We refer the reader to \cite{BS:2004,KS:2012,LM:1995,V:2003} and references therein for more on subshifts of finite type and their applications.

\subsection{Main results}

\noindent Throughout this article, following convention, we let the symbol $\mathbb{N}$ denote the set of natural numbers and $\mathbb{N}_{0}$ denote the set of non-negative integers.  Recall, from for instance \cite[Example 3.3.4]{DK:2002}, that the \textit{multinacci number} of order $n \in \mathbb{N} \setminus \{ 1 \}$ is the real number $\gamma_{n} \in (1, 2)$ which is the unique positive real solution of the equation $1 = x^{-1} + x^{-2} + \dots + x^{-n}$.  The smallest multinacci number is the multinacci number of order $2$ and is equal to the golden mean $(1+\sqrt{5})/2$.  We also observe that $\gamma_{n + 1} > \gamma_{n}$, for all $n \in \mathbb{N}$.  Further, for $n, k \in \mathbb{N}_{0}$ with $n \geq 2$, we define the algebraic integers $\beta_{n, k}$ and $\alpha_{n, k}$ by $(\beta_{n, k})^{2^{k}} = \gamma_{n}$ and $\alpha_{n, k} = 1 - \beta_{n, k}/2$.

\begin{theorem}\label{thm:SFTandTransitive}
For all $n,k \in \mathbb{N}$ with $n \geq 2$, we have the following.
\begin{enumerate}
\item\label{item:1:Density2} The intermediate $\beta_{n, k}$-transformations $T_{\beta_{n, k}, \alpha_{n, k}}^{\pm}$ are not transitive.
\item\label{item:2:Density2} The intermediate $\beta_{n, k}$-shift $\Omega_{\beta_{n, k}, \alpha_{n, k}}$ is a subshift of finite type.
\item\label{item:4:Density2} The greedy and lazy $\beta_{n, k}$-transformations are transitive.
\item\label{item:3:Density2} The greedy and lazy $\beta_{n, k}$-shifts, $\Omega_{\beta_{n, k},0}$ and $\Omega_{\beta_{n, k},2-\beta_{n, k}}$ respectively, are not sofic, and hence not a subshift of finite type.
\end{enumerate}
\end{theorem}

Part \ref{item:4:Density2} is well known and holds for all greedy, and hence lazy, $\beta$-transformations (see for instance \cite{G:1990,P:1979}).  It is included here both for completeness and to emphasise the dynamical differences which can occur between the greedy, and hence the lazy, \mbox{$\beta$-transformation} (shift) and an associated intermediate $\beta$-transformation (shift).

To verify Theorem~\ref{thm:SFTandTransitive}\ref{item:2:Density2} we apply Theorem~\ref{thm:ESFTP} given below.  This latter result is a generalisation of the following result of Parry \cite{P:1960}.  Here and in the sequel, $\tau^{\pm}_{\beta, \alpha}(p)$ denotes the points in the associated intermediate $\beta$-shifts which are a $\beta$-expansion of $p + \alpha/(\beta - 1)$, where $p = p_{\beta, \alpha}$.

\begin{theorem}[\cite{P:1960}]\label{thm:Parry1960}
For $\beta \in (1, 2)$, we have
\begin{enumerate}
\item\label{thmA:2} the greedy $\beta$-shift is a subshift of finite type if and only if $\tau^{-}_{\beta, 0}(p)$ is periodic and
\item\label{thmA:3} the lazy $\beta$-shift is a subshift of finite type if and only if $\tau^{+}_{\beta, 2-\beta}(p)$ is periodic.
\end{enumerate}
\end{theorem}

\begin{theorem}\label{thm:ESFTP}
Let $\beta \in (1, 2)$ and $\alpha \in (0, 2 - \beta)$.  The intermediate $\beta$-shift $\Omega_{\beta, \alpha}$ is a subshift of finite type if and only if both $\tau^{\pm}_{\beta, \alpha}(p)$ are periodic.
\end{theorem}

The following example demonstrates that it is possible to have that one of the sequences $\tau^{\pm}_{\beta, \alpha}(p)$ is periodic and that the other is not periodic.

\begin{example}
Let $\beta = \gamma_{2}$, $\alpha = 1/\beta^{4}$ and $p = (1 - \alpha)/\beta =1/\beta^{3}$.  Clearly $\alpha \in (0, 2 - \beta)$.  An elementary calculation will show that $\tau_{\beta, \alpha}^{-}(p)$ is the infinite periodic word with period $(0,1,1,0)$ and $\tau_{\beta, \alpha}^{+}(p)$ is the concatenation of the finite word $(1, 0, 0)$ with the infinite periodic word with period $(1,0)$.
\end{example}

For $\beta > 2$ and $\alpha \in [0, 1]$ the forward implications of Theorems~\ref{thm:Parry1960} and \ref{thm:ESFTP} can be found in, for instance, \cite[Theorem 6.3]{Will:1975}.  Further, for $\beta > 1$, necessary and sufficient conditions when $\Omega_{\beta, \alpha}$ is sofic shift is given in \cite[Theorem 2.14]{KS:2012}.  This latter result, in fact, gives that if $\beta \in (1, 2)$ is a Pisot number, then along the fiber
$\Delta(\beta) \coloneqq \{( b, \alpha) \in \Delta \colon b = \beta\}$ there exists a dense set of points $(\beta,\alpha)$ such that $\Omega_{\beta, \alpha}$ is a sofic shift.

This leads us to the final problem of studying how a greedy (or lazy) $\beta$-shift being a subshift of finite type is related to the corresponding intermediate $\beta$-shifts being a subshift of finite type.  We already know that it is possible to find intermediate $\beta$-shifts of finite type in the fibres $\Delta(\beta)$ even though the corresponding greedy and lazy $\beta$-shifts are not subshift of finite type. Given this, a natural question to ask is, can one determine when no intermediate $\beta$-shift is a subshift of finite type, for a given fixed $\beta$.  This is precisely what we address in the following result which is an almost immediate application of the characterisation provided in Theorems~\ref{thm:Parry1960} and \ref{thm:ESFTP}.

\begin{corollary}\label{cor:fibersSFT}
If $\beta \in (1, 2)$ is not the solution of any polynomial of finite degree with coefficients in $\{-1,0, 1\}$, then for all $\alpha \in [0, 2-\beta]$ the intermediate $\beta$-shift $\Omega_{\beta, \alpha}$ is not a subshift of finite type.
\end{corollary}

We observe that the values $\beta_{n, k}$, which are considered in Theorem~\ref{thm:SFTandTransitive}, are indeed a solution of a polynomial with coefficients in the set $\{ -1, 0, 1\}$ and thus do not satisfy the conditions of Corollary~\ref{cor:fibersSFT}; this is verified in the proof of Theorem~\ref{thm:SFTandTransitive}\ref{item:3:Density2}.

The phenomenon that the map $T_{\beta, \alpha}$ is not transitive but where the associated shift space $\Omega_{\beta,\alpha}$ is a subshift of finite type also appear in the setting of $(-\beta)$-transformations.  Indeed recently Ito and Sadahiro \cite{IS:2009} studied the expansions of a number in the base $(-\beta)$, with $\beta > 1$.  These expansions are related to a specific piecewise linear map with constant slope defined on the interval $[-\beta/(\beta + 1), 1/(\beta+ 1)]$ by $x \mapsto -\beta x - \lfloor -\beta x + \beta/(\beta + 1) \rfloor$.  (Here for $a \in \mathbb{R}$, we denote by $\lfloor a \rfloor$ the largest integer $n$ so that $n \leq a$.)  Such a map is called a \textit{negative $\beta$-transformation}.  The corresponding symbolic space space is called a \textit{($-\beta$)-shift} which is a subset of $\{ 0, 1\}^{\mathbb{N}}$ when $\beta < 2$.  By a result of Liao and Steiner \cite{LS:2012}, in contrast to the lazy or greedy $\beta$-shifts, for all $\beta < \gamma_{2}$, the negative $\beta$-transformations are not transitive.  On the other hand, a result of Frougny and Lai \cite{FL:2009} shows that a ($-\beta$)-shift is of finite type if and only if the representation of $-\beta/(\beta+1)$ in the ($-\beta$)-shift is periodic (see \eqref{eq:def_periodic} for the definition of a periodic word).  Thus one can construct examples of $\beta \in (1, 2)$ so that the negative $\beta$-transformation is not transitive but where the ($-\beta$)-shift is a subshift of finite type.  For instance, if we take the periodic sequence
\begin{align*}
(\overline{1, \underbrace{0, 0, \dots, 0}_{2k-\text{times}}, 1})
= (1, \underbrace{0, 0, \dots, 0}_{2k-\text{times}}, 1, 1, \underbrace{0, 0, \dots, 0}_{2k-\text{times}}, 1, 1, \underbrace{0, 0, \dots, 0}_{2k-\text{times}}, 1, \dots ),
\end{align*}
with $k \geq 1$, as the $(-\beta)$-expansion of $-\beta/(\beta + 1)$, then we obtain a $(-\beta)$-transformation which is not transitive and whose associated ($-\beta$)-shift is a subshift of finite type.  An interesting further problem would be to examine the corresponding questions for intermediate negative $\beta$-transformations and intermediate ($-\beta$)-shifts.  We refer the reader to \cite{HMP:2013} and references their in for further details on intermediate negative $\beta$-transformations and intermediate ($-\beta$)-shifts.

\subsection{Outline}

\noindent In Section~\ref{sec:Defn} we present basic definitions, preliminaries and auxiliary results required to prove Theorems~\ref{thm:SFTandTransitive} and \ref{thm:ESFTP} and Corollary~\ref{cor:fibersSFT}.  The proofs of Theorem~\ref{thm:ESFTP} and Corollary~\ref{cor:fibersSFT} are presented in Section~\ref{sec:TheoremA} and the proof of Theorem~\ref{thm:SFTandTransitive} is given in Section~\ref{sec:TheoremB}.\\\vspace{0.5in}

\section{Definitions and auxiliary results}\label{sec:Defn}

\subsection{Subshifts}\label{section:SFT}

\noindent We equipped the space $\{0,1\}^\mathbb{N}$ of infinite sequences with the topology induced by the word metric $d \colon \{0,1\}^\mathbb{N} \times \{0,1\}^\mathbb{N} \to \mathbb{R}$ which is given by
\begin{align*}
d(\omega, \nu) \coloneqq
\begin{cases}
0 & \text{if} \; \omega = \nu,\\
2^{- \lvert\omega \wedge \nu\rvert + 1} & \text{otherwise}.
\end{cases}
\end{align*}
Here $\rvert \omega \wedge \nu \lvert \coloneqq \min \, \{ \, n \in \mathbb{N} \colon \omega_{n} \neq \nu_n \}$, for all $\omega = (\omega_{1}, \omega_{2}, \dots) , \nu = ( \nu_{1} \, \nu_{2}, \dots) \in \{0, 1\}^{\mathbb{N}}$ with $\omega \neq \nu$. We let $\sigma$ denote the \textit{left-shift map} which is defined on the set
\begin{align*}
\{ 0, 1\}^{\mathbb{N}} \, \cup \, \{ \emptyset \} \, \cup \, \bigcup_{n = 1}^{\infty}\{ 0, 1\}^{n},
\end{align*}
of finite and infinite words over the alphabet $\{0, 1\}$ by $\sigma(\omega) \coloneqq \emptyset$, if $\omega \in \{ 0, 1\} \cup \{ \emptyset \}$, and otherwise we set  $\sigma(\omega_{1}, \omega_{2}, \dots) \coloneqq (\omega_{2}, \omega_{3}, \dots)$.  A \textit{shift invariant subspace} is a set $\Omega \subseteq \{0,1\}^\mathbb{N}$ such that $\sigma(\Omega) \subseteq \Omega$ and a \textit{subshift} is a closed shift invariant subspace.

Given a subshift $\Omega$ and $n \in \mathbb{N}$ set
\begin{align*}
\Omega\lvert_{n} \coloneqq \bigg\{ (\omega_{1}, \dots, \omega_{n}) \in \{ 0, 1\}^{n} \colon \,\text{there exists} \; (\xi_{1},  \dots ) \in \Omega \; \text{with} \; (\xi_{1}, \dots, \xi_{n}) = (\omega_{1}, \dots, \omega_{n}) \bigg\}
\end{align*}
and write $\Omega\lvert^{*} \coloneqq \bigcup_{k = 1}^{\infty} \Omega\lvert_{k}$ for the collection of all finite words.

A subshift $\Omega$ is a called a \textit{subshift of finite type} (SFT) if there exists an $M \in \mathbb{N}$ such that, for all $\omega = (\omega_{1}, \omega_{2}, \dots, \omega_{n}), \xi = (\xi_{1}, \xi_{2}, \dots, \xi_{m}) \in \Omega\lvert^{*}$ with $n, m \in \mathbb{N}$ and $n \geq M$,
\[
(\omega_{n - M + 1}, \omega_{n - M + 2}, \dots, \omega_{n}, \xi_{1}, \xi_{2}, \dots, \xi_{m}) \in \Omega\lvert^{*}
\]
if and only if
\[
(\omega_{1}, \omega_{2}, \dots, \omega_{n}, \xi_{1}, \xi_{2}, \dots, \xi_{m}) \in \Omega\lvert^{*}.
\]
The following result gives an equivalent condition for when a subshift is a subshift of finite type (see, for instance, \cite[Theorem 2.1.8]{LM:1995} for a proof of the equivalence).  For this we require the following notation.   For $n \in \mathbb{N}$ and $\omega = (\omega_{1}, \omega_{2}, \dots) \in \{ 0, 1\}^{\mathbb{N}}$, we set $\omega\lvert_{n} \coloneqq (\omega_{1}, \omega_{2}, \dots, \omega_{n})$.  Also, for a finite word $\xi \in \{ 0, 1 \}^{n}$, for some $n \in \mathbb{N}_{0}$, we let $\lvert \xi \rvert$ denote the length of $\xi$, namely the value $n$.

\begin{theorem}[{\cite{LM:1995}}]\label{thm:equivalent-def-SFT}
The space $\Omega$ is a SFT if and only if there exists a finite set $F$ of finite words in the alphabet $\{0, 1\}$ such that if $\xi \in F$, then $\sigma^{m}(\omega)\lvert_{\lvert \xi \rvert} \neq \xi$, for all $\omega \in \Omega$ and $m \in \mathbb{N}_{0}$.
\end{theorem}

The set $F$ in Theorem~\ref{thm:equivalent-def-SFT} is referred to as the \textit{set of forbidden words}.  Further, a subshift $\Omega$ is called \textit{sofic} if and only if it is a factor of a SFT.

For $\nu = (\nu_{1}, \nu_{2}, \dots, \nu_{n}) \in \{ 0, 1\}^{n}$ and $\xi = (\xi_{1}, \xi_{2}, \dots, \xi_{m}) \in \{ 0, 1\}^{m}$, denote by $(\nu, \xi)$ the concatenation $(\nu_{1}, \nu_{2}, \dots, \nu_{n}, \xi_{1}, \xi_{2}, \dots, \xi_{m}) \in \{ 0, 1\}^{n+m}$, for $n, m \in \mathbb{N}$.  We use the same notation when $\xi \in \{ 0, 1\}^{\mathbb{N}}$.

An infinite sequence $\omega = (\omega_{1}, \omega_{2}, \dots) \in \{0, 1\}^{\mathbb{N}}$ is called \textit{periodic} if there exists an $n \in \mathbb{N}$ such that 
\begin{align}\label{eq:def_periodic}
(\omega_{1}, \omega_{2}, \dots, \omega_{n}) = (\omega_{(m - 1)n + 1}, \omega_{(m - 1)n + 2}, \dots, \omega_{m n}),
\end{align}
for all $m \in \mathbb{N}$.  We call the least such $n$ the period of $\omega$ and we write $\omega = (\overline{\omega_{1}, \omega_{2}, \dots, \omega_{n}})$.  Similarly, $\omega = (\omega_{1}, \omega_{2}, \dots) \in \{0, 1\}^{\mathbb{N}}$ is called \textit{eventually periodic} if there exists an $n, k \in \mathbb{N}$ such that
\begin{align*}
(\omega_{k+1}, \dots, \omega_{k+n}) = (\omega_{k+(m - 1)n + 1}, \omega_{k+(m - 1)n + 2}, \dots, \omega_{k+ m n}),
\end{align*}
for all $m \in \mathbb{N}$, and we write $\omega = (\omega_{1}, \omega_{2}, \dots, \omega_{k-1}, \omega_{k}, \overline{\omega_{k+1}, \dots, \omega_{k+n}})$.

\subsection{Intermediate $\beta$-shifts and expansions}

\noindent We now give the definition of the intermediate $\beta$-shift. Throughout this section we let $(\beta, \alpha) \in \Delta$ be fixed and let  $p$ denote the value  $p_{\beta,\alpha} = (1 - \alpha)/\beta$.

The \textit{$T^{+}_{\beta, \alpha}$-expansion} $\tau_{\beta, \alpha}^{+}(x)$ of $x \in [0, 1]$ and the \textit{$T^{-}_{\beta, \alpha}$-expansion} $\tau_{\beta, \alpha}^{-}(x)$ of $x\in [0, 1]$ are respectively defined to be $\omega^{+} = (\omega^{+}_{1}, \omega^{+}_{2}, \dots, ) \in \{ 0, 1\}^{\mathbb{N}}$ and $\omega^{-} = (\omega^{-}_{1}, \omega^{-}_{2}, \dots, ) \in \{ 0, 1\}^{\mathbb{N}}$, where
\begin{align*}
\omega^{-}_{n} \coloneqq \begin{cases}
0, & \quad \text{if } \, (T^{-}_{\beta,\alpha})^{n-1}(x) \leq p,\\
1, & \quad \text{otherwise,}
\end{cases}
\quad
\text{and}
\quad
\omega^{+}_{n} \coloneqq \begin{cases}
0, & \quad \text{if } (T^{+}_{\beta,\alpha})^{n-1}(x) < p,\\
1, & \quad \text{otherwise,}
\end{cases}
\end{align*}
for all $n \in \mathbb{N}$.  We will denote the images of the unit interval under $\tau_{\beta, \alpha}^{\pm}$ by $\Omega^{\pm}_{\beta, \alpha}$, respectively, and write $\Omega_{\beta, \alpha}$ for the union $\Omega_{\beta, \alpha}^{+} \cup \Omega_{\beta, \alpha}^{-}$.  The \textit{upper} and \textit{lower kneading invariants} of $T_{\beta,\alpha}^{\pm}$ are defined to be the infinite words $\tau^{\pm}_{\beta, \alpha}(p)$, respectively, and will turn out to be of great importance.

\begin{remark}\label{rem:00..11..}
Let $\omega^{\pm} = (\omega_{1}^{\pm}, \omega_{2}^{\pm}, \dots, )$ denote the infinite words $\tau_{\beta, \alpha}^{\pm}(p)$, respectively.  By definition, $\omega^{-}_{1} = \omega^{+}_{2} = 0$ and $\omega^{+}_{1} = \omega^{-}_{2} = 1$.  Furthermore, $(\omega_{k}^{+}, \omega_{k+1}^{+}, \dots, ) = (0, 0, 0, \dots )$ if and only if $k = 2$ and $\alpha = 0$; and $(\omega_{k}^{-}, \omega_{k+1}^{-}, \dots, ) = (1, 1, 1, \dots )$ if and only if $k = 2$ and $\alpha = 2-\beta$.
\end{remark}

The connection between the maps $T_{\beta, \alpha}^{\pm}$ and the $\beta$-expansions of real numbers is given through the $T_{\beta, \alpha}^{\pm}$-expansions of a point and the projection map $\pi_{\beta, \alpha} \colon \{ 0, 1 \}^{\mathbb{N}} \to [0, 1]$, which we will shortly define.  The projection map $\pi_{\beta, \alpha}$ is linked to the underlying (overlapping) iterated function system (IFS), namely $( [0, 1]; f_{0} \colon x \to x/\beta, f_{1} \colon x \to x/\beta + 1 - 1/\beta )$, where $\beta \in (1, 2)$.  (We refer the reader to \cite{F:1990} for the definition of and further details on IFSs.)  The projection map $\pi_{\beta, \alpha}$ is defined by
\begin{align*}
\pi_{\beta, \alpha}(\omega_{1}, \omega_{2}, \dots) \coloneqq (\beta - 1)^{-1} (\lim_{n \to \infty} f_{\omega_{1}} \circ \dots \circ f_{\omega_{n}}([ 0, 1]) - \alpha) =  \alpha(1 - \beta)^{-1} + \sum_{k = 1}^{\infty} \omega_{k} \beta^{-k}.
\end{align*}
An important property of the projection map is that the following diagram commutes.
\begin{align}\label{diag:commutative}
\begin{array}
[c]{ccc}
\Omega_{\beta, \alpha}^{\pm} & \overset{\sigma}{\longrightarrow} & \Omega_{\beta, \alpha}^{\pm}\\
& & \\
\pi_{\beta, \alpha} \downarrow\text{\ \ \ \ } &  & \text{ \ \ \ }\downarrow\pi_{\beta, \alpha} \\ & & \\
 \lbrack 0,1 \rbrack & \underset{T_{\beta, \alpha}^{\pm}} {\longrightarrow} & \lbrack 0,1\rbrack
\end{array}
\end{align}
(This result is readily verifiable from the definitions of the maps involved and a sketch of the proof of this result can be found in \cite[Section~5]{BHV:2011}.)  From this, we conclude that, for each $x \in [\alpha/(\beta - 1), 1 + \alpha/(\beta-1)]$, a $\beta$-expansion of $x$ is the infinite word $\tau^{\pm}_{\beta, \alpha}(x - \alpha (\beta - 1)^{-1})$.

\subsection{The extended model}

\noindent In our proof of Theorem \ref{thm:ESFTP}, we will consider the \textit{extended model} given as follows.  For $(\beta, \alpha) \in \Delta$, we let $p = p_{\beta,\alpha} = (1 - \alpha)/\beta$ and we define the transformation $\widetilde{T}_{\beta, \alpha}^{\pm} \colon [-\alpha/(\beta - 1), (1- \alpha)/(\beta-1)] \circlearrowleft$ by
\begin{align*}
\widetilde{T}_{\beta, \alpha}^{+}(x) = \begin{cases}
\beta x + \alpha & \quad\text{if} \; x < p,\\
\beta x + \alpha - 1 & \quad\text{otherwise,}\\
\end{cases}
\qquad \text{and} \qquad
\widetilde{T}_{\beta, \alpha}^{-}(x) = \begin{cases}
\beta x + \alpha & \quad\text{if} \; x \leq p,\\
\beta x + \alpha - 1 & \quad\text{otherwise}\\
\end{cases}
\end{align*}
(see Figure~\ref{Fig4}).  Note that $\widetilde{T}_{\beta, \alpha}^{\pm}\lvert_{[0, 1]} = T_{\beta, \alpha}^{\pm}$ and that, for all points $x$ chosen from the interval $(-\alpha/(\beta - 1), 0) \cup (1, (1- \alpha)/(1-\beta))$, there exists a minimal $n = n(x) \in \mathbb{N}$ such that we have $(\widetilde{T}_{\beta, \alpha}^{\pm})^{m}(x) \in [0, 1]$, for all $m \geq n$.
\begin{figure}[htbp]
\centering
\subfloat[Plot of $\widetilde{T}_{\beta,\alpha}^{+}$.]{
\hspace{1em}
\scalebox{0.45}{
\includegraphics{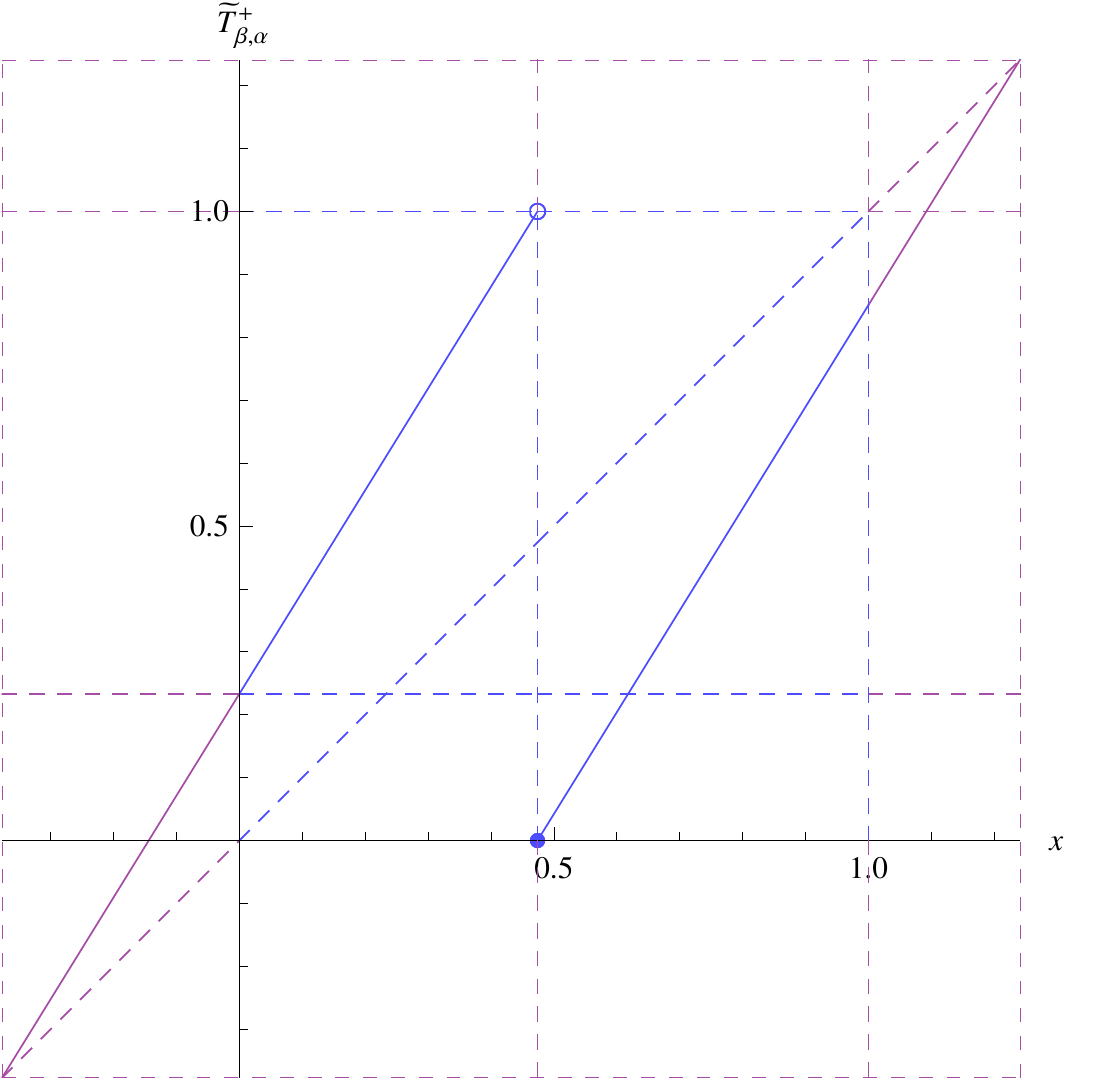}	
	}}
	\hspace{2em}
\subfloat[Plot of $\widetilde{T}_{\beta, \alpha}^{-}$.]{
\scalebox{0.45}{
\includegraphics{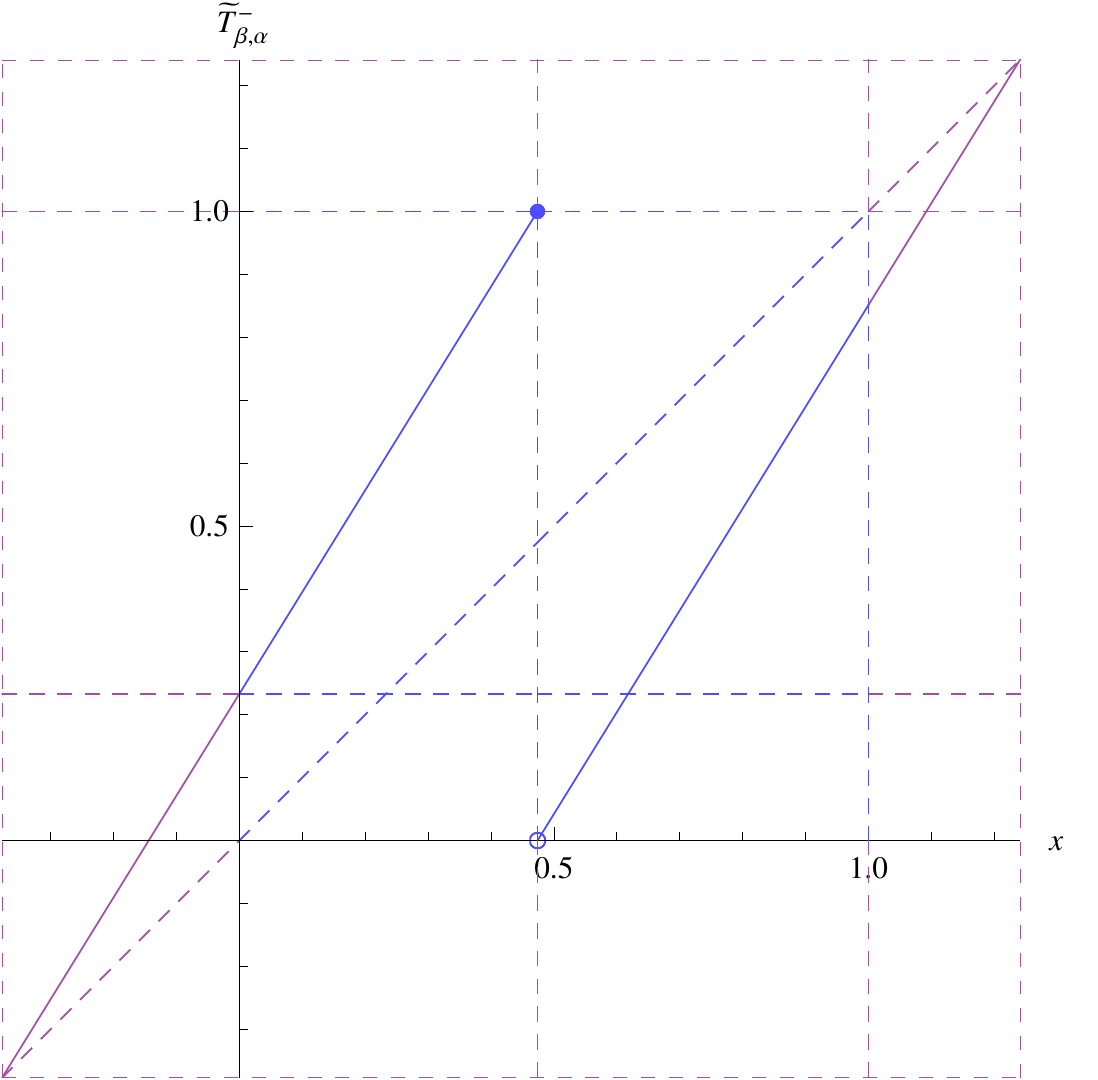}	
	}}
	\caption{Plot of $\widetilde{T}^{\pm}_{\beta, \alpha}$ for $\beta = (\sqrt{5} + 1)/2$ and $\alpha = 1 - 0.474 \beta$.  (The height of the filled in circle determines the value of the map at the point of discontinuity.)}
	\label{Fig4}
\end{figure}

The $\widetilde{T}_{\beta, \alpha}^{\pm}$-expansions $\widetilde{\tau}^{\pm}_{\beta, \alpha}(x)$ of points $x \in [-\alpha/(\beta - 1), (1- \alpha)/(\beta-1)]$ is defined in the same way as $\tau_{\beta, \alpha}^{\pm}$ except one replaces $T_{\beta, \alpha}^{\pm}$ by  $\widetilde{T}_{\beta, \alpha}^{\pm}$.  We let 
$$\widetilde{\Omega}_{\beta, \alpha}^{\pm} \coloneqq \widetilde{\tau}^{\pm}_{\beta, \alpha}([-\alpha/(\beta - 1), (1- \alpha)/(\beta-1)])$$ 
and we set $\widetilde{\Omega}_{\beta, \alpha} = \widetilde{\Omega}_{\beta, \alpha}^{+} \cup \widetilde{\Omega}_{\beta, \alpha}^{-}$.  The \textit{upper} and \textit{lower kneading invariant} of $\widetilde{T}_{\beta, \alpha}^{\pm}$ are defined to be the infinite words  $\widetilde{\tau}^{\pm}_{\beta, \alpha}(p)$, respectively.

\subsection{Structure of $\Omega_{\beta, \alpha}^{\pm}$ and $\widetilde{\Omega}_{\beta, \alpha}^{\pm}$}

\noindent The following theorem, on the structure of the spaces $\Omega_{\beta, \alpha}^{\pm}$ and $\widetilde{\Omega}_{\beta, \alpha}^{\pm}$, will play a crucial role in the proof of  Theorem~\ref{thm:ESFTP}.  A partial version of this result can be found in \cite[Lemma 1]{H:1979}.  To the best of our knowledge, this first appeared in \cite[Theorem 1]{HS:1990}, and later in \cite{AM:1996,BM:2011,BHV:2011,KS:2012}.  Moreover, a converse of Theorem~\ref{thm:Structure} holds true and was first given in \cite[Theorem 1]{HS:1990} and later in \cite[Theorem 1]{BV:2012}.

\begin{theorem}[{\cite{AM:1996,BM:2011,BV:2012,BHV:2011,H:1979,HS:1990,KS:2012}}]\label{thm:Structure}
For $(\beta, \alpha) \in \Delta$, the spaces $\Omega_{\beta, \alpha}^{\pm}$ and $\widetilde{\Omega}_{\beta, \alpha}^{\pm}$ are completely determined by upper and lower kneading invariants of $T_{\beta, \alpha}^{\pm}$ and $\widetilde{T}_{\beta, \alpha}^{\pm}$, respectively:
\begin{align*}
\Omega_{\beta, \alpha}^{+} &= \left\{ \omega \in \{ 0, 1\}^{\mathbb{N}} \colon \forall n \geq 0, \tau_{\beta, \alpha}^{+}(0) \preceq \sigma^{n}(\omega) \prec \tau_{\beta, \alpha}^{-}(p) \; \textup{or} \; \tau_{\beta, \alpha}^{+}(p) \preceq \sigma^{n}(\omega) \preceq \tau_{\beta, \alpha}^{-}(1) \right\}\\
\Omega_{\beta, \alpha}^{-} &= \left\{ \omega \in \{ 0, 1\}^{\mathbb{N}} \colon \forall \, n \geq 0, \tau_{\beta, \alpha}^{+}(0) \preceq  \sigma^{n}(\omega) \preceq \tau_{\beta, \alpha}^{-}(p) \; \textup{or} \; \tau_{\beta, \alpha}^{+}(p) \prec \sigma^{n}(\omega) \preceq \tau_{\beta, \alpha}^{-}(1) \right\}\\
\widetilde{\Omega}_{\beta, \alpha}^{+} &= \left\{ \omega \in \{ 0, 1\}^{\mathbb{N}} \colon \forall \, n \geq 0, \widetilde{\tau}_{\beta, \alpha}^{+}\left(\textstyle{\frac{-\alpha}{\beta - 1}}\right) \preceq \sigma^{n}(\omega) \prec \widetilde{\tau}_{\beta, \alpha}^{-}(p) \; \textup{or} \; \widetilde{\tau}_{\beta, \alpha}^{+}(p) \preceq \sigma^{n}(\omega) \preceq \widetilde{\tau}_{\beta, \alpha}^{+} \left(\textstyle{\frac{1- \alpha}{\beta-1}}\right) \right\}\\
\widetilde{\Omega}_{\beta, \alpha}^{-} &= \left\{ \omega \in \{ 0, 1\}^{\mathbb{N}} \colon \forall \, n \geq 0, \widetilde{\tau}_{\beta, \alpha}^{+}\left(\textstyle{\frac{-\alpha}{\beta - 1}}\right) \preceq \sigma^{n}(\omega) \preceq \widetilde{\tau}_{\beta, \alpha}^{-}(p) \; \textup{or} \; \widetilde{\tau}_{\beta, \alpha}^{+}(p) \prec \sigma^{n}(\omega) \preceq \widetilde{\tau}_{\beta, \alpha}^{+}\left(\textstyle{\frac{1- \alpha}{\beta-1}}\right) \right\}.
\end{align*}
(Here the symbols $\prec$, $\preceq$, $\succ$ and $\succeq$ denote the lexicographic orderings on $\{ 0 ,1\}^{\mathbb{N}}$.)  Moreover,
\begin{align}\label{eq:closure-union}
\Omega_{\beta, \alpha} = \Omega_{\beta, \alpha}^{+} \cup \Omega_{\beta, \alpha}^{-} \quad \text{and} \quad \widetilde{\Omega}_{\beta, \alpha} = \widetilde{\Omega}_{\beta, \alpha}^{+} \cup \widetilde{\Omega}_{\beta, \alpha}^{-}.
\end{align}
\end{theorem}

\begin{remark}\label{rmk:tildeVnon-tilde}
By the commutativity of the diagram in (\ref{diag:commutative}), we have that $\tau_{\beta, \alpha}^{+}(0) = \sigma(\tau_{\beta, \alpha}^{+}(p))$ and $\tau_{\beta, \alpha}^{-}(1) = \sigma(\tau_{\beta, \alpha}^{-}(p))$.  Additionally, from the definition of the $\widetilde{T}^{\pm}_{\beta,\alpha}$-expansion of a point
\begin{align*}
\widetilde{\tau}_{\beta, \alpha}^{+}\left(\textstyle{\frac{-\alpha}{\beta - 1}}\right) = \widetilde{\tau}_{\beta, \alpha}^{-}\left(\textstyle{\frac{-\alpha}{\beta - 1}}\right) = (\overline{0})
\quad
\text{and}
\quad
\widetilde{\tau}_{\beta, \alpha}^{+}\left(\textstyle{\frac{1- \alpha}{\beta-1}}\right) = \widetilde{\tau}_{\beta, \alpha}^{-}\left(\textstyle{\frac{1- \alpha}{\beta-1}}\right) = (\overline{1}).
\end{align*}
Further, since  $\widetilde{T}_{\beta, \alpha}^{\pm}\lvert_{[0, 1]} = T_{\beta, \alpha}^{\pm}$, it follows that $\tau_{\beta, \alpha}^{\pm}(p) = \widetilde{\tau}_{\beta, \alpha}^{\pm}(p)$ and thus, from now on we will write $\tau_{\beta, \alpha}^{+}(p)$ for the common value $\tau_{\beta, \alpha}^{+}(p) = \widetilde{\tau}_{\beta, \alpha}^{+}(p)$ and $\tau_{\beta, \alpha}^{-}(p)$ for the common value $\tau_{\beta, \alpha}^{-}(p) = \widetilde{\tau}_{\beta, \alpha}^{-}(p)$.
\end{remark}

\begin{corollary}\label{cor:shift_invarance}
For $(\beta, \alpha) \in \Delta$, we have that $\Omega_{\beta, \alpha}^{\pm}$ and $\widetilde{\Omega}_{\beta, \alpha}^{\pm}$ are shift invariant subspaces and that $\Omega_{\beta, \alpha}$ and $\widetilde{\Omega}_{\beta, \alpha}$ are subshifts.  Moreover,
\begin{align*}
\sigma(\Omega_{\beta, \alpha}\lvert^{*}) \subseteq \Omega_{\beta, \alpha}\lvert^{*},
\quad
\sigma(\widetilde{\Omega}_{\beta, \alpha}\lvert^{*}) \subseteq \widetilde{\Omega}_{\beta, \alpha}\lvert^{*},
\quad
\sigma(\Omega_{\beta, \alpha}) = \Omega_{\beta, \alpha}
\quad \text{and} \quad
\sigma(\widetilde{\Omega}_{\beta, \alpha}) = \widetilde{\Omega}_{\beta, \alpha}.
\end{align*}
\end{corollary}

\begin{proof}
This is a direct consequence of Theorem~\ref{thm:Structure} and the commutativity of the diagram given in (\ref{diag:commutative}).
\end{proof}

As one can see the code space structure of $\widetilde{\Omega}_{\beta, \alpha}^{\pm}$ is simpler than the code space structure of $\Omega_{\beta, \alpha}^{\pm}$.  In fact, to prove Theorem~\ref{thm:ESFTP}, we will show that it is necessary and sufficient to show that $\widetilde{\Omega}_{\beta, \alpha}$ is a SFT.

\subsection{Periodicity and zero/one-full words}

\noindent We now give a sufficient condition for the kneading invariants of $\widetilde{T}^{\pm}_{\beta, \alpha}$ (and hence, by Remark~\ref{rmk:tildeVnon-tilde}, of $T^{\pm}_{\beta, \alpha}$) to be periodic.  For this, we will require the following notations and definitions.  Indeed by considering zero and one-fullness (Definition~\ref{defn:fullness}) of the infinite words $\sigma ( \tau_{\beta, \alpha}^{\pm}(p))$ we obtain a close link to the definition of SFT (see Lemma~\ref{lem:keylemma} and Proposition~\ref{prop:ESFTP}\ref{propA:1}).

For a given subset $E$ of $\mathbb{N}$, we write $\lvert E \rvert$ for the cardinality of $E$.

\begin{definition}\label{defn:fullness}
Let $(\beta, \alpha) \in \Delta$.  An infinite word $\omega = (\omega_{1}, \omega_{2}, \dots) \in \{ 0, 1\}^{\mathbb{N}}$ is called
\begin{enumerate}[label*=(\alph*)]
\item \textit{zero-full} with respect to $(\beta, \alpha)$ if
\begin{align*}
N_{0}(\omega) \coloneqq \left\lvert \left\{ k > 2 \colon (\omega_{1}, \omega_{2}, \dots, \omega_{k-1}, 1) \in \widetilde{\Omega}^{-}_{\beta, \alpha}\lvert_{k} \; \text{and} \; \omega_{k} = 0 \right\} \right\rvert,
\end{align*}
is not finite, and 
\item \textit{one-full} with respect to $(\beta, \alpha)$ if
\begin{align*}
N_{1}(\omega) \coloneqq \left\lvert \left\{ k > 2 \colon (\omega_{1}, \omega_{2}, \dots, \omega_{k-1}, 0) \in \widetilde{\Omega}^{+}_{\beta, \alpha}\lvert_{k} \; \text{and} \; \omega_{k} = 1 \right\} \right\rvert,
\end{align*}
is not finite. 
\end{enumerate}
\end{definition}

\begin{remark}
The concept of an infinite word being zero-full is motivated from the fact that a cylinder set $\mathrm{I}(\omega_1,\dots,\omega_{n-1},0) \subset [0, 1)$ corresponding to the finite word $(\omega_1,\dots,\omega_{n-1},0)$ is a full interval if  $(\omega_1,\dots,\omega_{n-1},1)$ is admissible for the greedy $\beta$-expansion with $1 < \beta < 2$ (see for instance \cite[Lemma 3.2 (1)]{FW:2012}).  Here we recall from \cite{FW:2012} that
\begin{enumerate}
\item the cylinder set $\mathrm{I}(\omega_1,\dots,\omega_{n-1},0) \subset [0, 1)$ is defined to be the half open interval of $[0, 1]$ such that for any $x \in \mathrm{I}(\omega_1,\dots,\omega_{n-1},0)$ we have $\tau^{-}_{\beta, 0}(x)\vert_{n} = (\omega_1,\dots,\omega_{n-1},0)$,
\item \textit{admissible for the greedy $\beta$-expansion} means that $\sigma^{k}(\omega_1,\dots,\omega_{n-1},1) \preceq \tau_{\beta, 0}^{-}(1)\vert_{n-1-k}$, for all $k \in \{ 0, 1, \dots, n-2 \}$, and
\item \textit{full} means that the length of the cylinder set $\mathrm{I}(\omega_1,\dots,\omega_{n-1},0)$ is equal to $\beta^{-n}$, which is equivalent to $(T_{\beta, 0}^{-})^{n}(\mathrm{I}(\omega_1,\dots,\omega_{n-1},0)) = [0, 1)$.
\end{enumerate}
The notion of an infinite word being one-full is motivated as above except from the prospective of the lazy $\beta$-expansion instead of the greedy $\beta$-expansion.
\end{remark}

Before stating our next result we remark that $p_{\beta, \alpha} = 1 - 1/\beta$ if and only if $\alpha = 2-\beta$ and $p_{\beta, \alpha} = 1/\beta$ if and only if $\alpha = 0$.

\begin{lemma}\label{lem:keylemma}
Let $(\beta, \alpha) \in \Delta$ and set $p = p_{\beta, \alpha}$.
\begin{enumerate}
\item\label{keylemma1} If $\tau_{\beta, \alpha}^{+}(p)$ is not periodic and $\alpha \neq 2 - \beta$, then $\sigma ( \tau_{\beta, \alpha}^{+}(p))$ is one-full. 
\item\label{keylemma2} If $\tau_{\beta, \alpha}^{-}(p)$ is not periodic and $\alpha \neq 0$, then $\sigma ( \tau_{\beta, \alpha}^{-}(p))$ is zero-full.
\end{enumerate}
\end{lemma}

\begin{proof}
As the proofs for \ref{keylemma1} and \ref{keylemma2} are essentially the same, we only include a proof of \ref{keylemma2}.  For ease of notation let $\nu = (\nu_{1}, \nu_{2}, \dots )$ denote the infinite word $\sigma(\tau_{\beta, \alpha}^{-}(p))$ and, note that, since $\tau_{\beta, \alpha}^{-}(p)$ is not periodic, $\sigma^{k}(\tau_{\beta,\alpha}^{-}(p)) \neq \tau_{\beta, \alpha}^{-}(p)$, for all $k \in \mathbb{N}$.

Define the non-constant monotonic sequence $(n_{k})_{k \in \mathbb{N}}$ of natural numbers as follows.  Let $n_{1} \in \mathbb{N}$ be the least natural number such that $\nu_{n_{1}-1} = 1$ and $\nu_{n_{1}} = 0$ and, for all $k > 1$, let $n_{k} \in \mathbb{N}$ be such that
\begin{align}\label{lem:key lemma:proof:con}
\sigma^{n_{k-1}-1}(\nu)\lvert_{n_{k}-n_{k-1}} = \tau_{\beta, \alpha}^{-}(p)\lvert_{n_{k} - n_{k - 1}} \quad \text{and} \quad \sigma^{n_{k-1}-1}(\nu)\lvert_{n_{k}-n_{k-1}+1} \prec \tau_{\beta, \alpha}^{-}(p)\lvert_{n_{k} - n_{k - 1}+1}.
\end{align}
Such a sequence exists by Remark~\ref{rem:00..11..} and because $\tau_{\beta, \alpha}^{-}(p)$ is assumed to be not periodic.  We will now show that $\nu_{n_{k}} = 0$ and that
\begin{align*}
(\nu_{1}, \nu_{2}, \dots, \nu_{n_{k}-1},1) \in \widetilde{\Omega}_{\beta, \alpha}^{-}\lvert_{n_{k}},
\end{align*}
for all integers $k > 1$.  By the orderings given in (\ref{lem:key lemma:proof:con}) it immediately follows that $\nu_{n_{k}} = 0$.  In order to conclude the proof, since $\nu \in \widetilde{\Omega}_{\beta, \alpha}^{-}$ and since $\tau_{\beta, \alpha}^{-}(p)$ is not periodic, by Theorem~\ref{thm:Structure}, it is sufficient to show, for all integers $k > 1$ and $m \in \{ 0, 1, 2, \dots, n_{k} -2, n_{k} -1\}$, that
\begin{align*}
\sigma^{m}(\theta) \;\; \begin{cases}
\;\; \preceq \;\; \tau_{\beta, \alpha}^{-}(p) & \quad\text{if} \; \nu_{m + 1} = 0,\\
\;\; \succ \;\; \tau_{\beta, \alpha}^{+}(p) & \quad\textit{otherwise,}
\end{cases}
\end{align*}
where $\theta \coloneqq (\nu_{1}, \nu_{2}, \dots, \nu_{n_{k}-1}, 1, \nu_{n_{k}-n_{k-1}+1}, \nu_{n_{k}-n_{k-1}+2}, \dots)$.

Suppose that $l \in \{1, 2, \dots, k\}$ and $m \in \{ n_{l -1}, n_{l-1}+1, \dots, n_{l}-2 \}$, where $n_{0} \coloneqq 0$.  If $\nu_{m+1} = 1$, then
\begin{align*}
\sigma^{m}(\theta) &= (\nu_{m+1}, \nu_{m+2}, \dots, \nu_{n_{l}-1}, \nu_{n_{l}}, \nu_{n_{l}+1}, \dots, \nu_{n_{k}-1}, 1, \nu_{n_{k}-n_{k-1}+1}, \nu_{n_{k}-n_{k-1}+2}, \dots )\\
&\succ (\nu_{m+1}, \nu_{m+2}, \dots, \nu_{n_{l}-1}, \nu_{n_{l}}, \nu_{n_{l}+1}, \dots, \nu_{n_{k}-1}, 0, \nu_{n_{k} + 1}, \nu_{n_{k} + 2}, \dots )\\
&= \sigma^{m}(\nu)\\
&\succ \tau_{\beta, \alpha}^{+}(p).
\end{align*}
Now suppose $v_{m+1} = 0$.  By \eqref{lem:key lemma:proof:con}, we have $(\nu_{n_{l-1}}, \nu_{n_{l-1} + 1},  \dots, \nu_{n_{l} - 1}, 1) = \tau_{\beta, \alpha}^{-}(p)\lvert_{n_l-n_{l-1}+1}$, and so,
\begin{align*}
\sigma^{m-n_{l-1}+1}(\tau_{\beta, \alpha}^{-}(p))
= (\nu_{m + 1},\nu_{m + 2}, \dots, \nu_{n_{l} - 1}, 1, \nu_{n_{l}-n_{l-1}+1}, \nu_{n_{l}-n_{l-1}+2}, \dots).
\end{align*}
Let us first consider the case $l = k$. Since by assumption $v_{m+1} = 0$, and since, by Theorem~\ref{thm:Structure}, we have that $\sigma^{m-n_{l-1}+1}(\tau_{\beta, \alpha}^{-}(p)) \preceq \tau_{\beta, \alpha}^{-}(p)$, it follows that
\begin{align*}
\sigma^{m}(\theta) = (\nu_{m + 1},\nu_{m + 2}, \dots, \nu_{n_{l} - 1}, 1, \nu_{n_{l}-n_{l-1}+1}, \nu_{n_{l}-n_{l-1}+2}, \dots) \preceq \tau_{\beta, \alpha}^{-}(p).
\end{align*}
Therefore, in this case, namely $l = k$, the result follows.  Now assume that $l < k$.  In this case we have that
\begin{align*}
\sigma^{m}(\theta) &= (\nu_{m + 1},\nu_{m + 2},  \dots, \nu_{n_{l} - 1},\nu_{n_l}, \nu_{n_{l} + 1}, \dots, \nu_{n_{k}-1}, 1, \nu_{n_{k}-n_{k-1}+1}, \nu_{n_{k}-n_{k-1}+2}, \dots)\\
&\prec (\nu_{m + 1}, \nu_{m + 2}, \dots, \nu_{n_{l} - 1}, 1, \nu_{n_{l}-n_{l-1}+1}, \nu_{n_{l}-n_{l-1}+2}, \dots)\\
&\preceq \tau_{\beta, \alpha}^{-}(p).
\end{align*}
If $l \in \{1, 2, \dots, k-2\}$ and $m = n_{l} -1$, then by the orderings given in (\ref{lem:key lemma:proof:con}),
\begin{align*}
\sigma^{m}(\theta) = (\underbrace{\nu_{n_{l}}, \dots, \nu_{n_{l+1}}}_{\prec \tau_{\beta, \alpha}^{-}(p)\lvert_{n_{l+1}-n_{l}+1}}, \dots, \nu_{n_{k}-1}, 1, \nu_{n_{k}-n_{k-1}+1}, \nu_{n_{k}-n_{k-1}+2}, \dots ) \prec \tau_{\beta, \alpha}^{-}(p).
\end{align*}
If  $m = n_{k - 1} -1$, then by the orderings given in (\ref{lem:key lemma:proof:con}),
\begin{align*}
\sigma^{m}(\theta)&= (\underbrace{\nu_{n_{k-1}}, \dots, \nu_{n_{k}-1}, 1}_{= \tau_{\beta, \alpha}^{-}(p)\lvert_{n_{k}-n_{k-1}+1}}, \underbrace{\nu_{n_{k}-n_{k-1}+1} \nu_{n_{k}-n_{k-1}+2}, \dots}_{=\sigma^{n_{k}-n_{k-1}+1}(\tau_{\beta, \alpha}^{-}(p))} ) =  \tau_{\beta, \alpha}^{-}(p).
\end{align*}
Finally, if $m = n_{k}-1$, then either $\nu_{n_{k}-n_{k-1}+1} =1$, in which case the result follows from Remark~\ref{rem:00..11..} or else $\nu_{n_{k}-n_{k-1}+1} =0$ in which case
\begin{align*}
\sigma^{m}(\theta)&= (1, \nu_{n_{k}-n_{k-1}+1} \nu_{n_{k}-n_{k-1}+2}, \dots ) = \sigma^{n_{k}-n_{k-1}}(\tau_{\beta, \alpha}^{-}(p)) \succ \tau_{\beta, \alpha}^{+}(p).
\end{align*}
This concludes the proof.
\end{proof}

\subsection{Symmetric intermediate $\beta$-shifts}

\noindent The following remarks and result on symmetric intermediate $\beta$-shifts will assist us in shortening the proof of Theorem~\ref{thm:ESFTP} and the proof of Theorem~\ref{thm:SFTandTransitive}\ref{item:2:Density2}.  For this, let us recall that two maps $R \colon X \circlearrowleft$ and $S \colon Y \circlearrowleft$ defined on compact metric spaces are called \textit{topologically conjugate} if there exists a homeomorphism $h \colon X \to Y$ such that $S \circ h = h \circ R$.

The maps $T_{\beta, \alpha}^{-}$ and $T_{\beta, 2 - \alpha - \beta}^{+}$ are topologically conjugate and the maps $T_{\beta, \alpha}^{-}$ and $T_{\beta, 2-\alpha-\beta}^{+}$ are topologically conjugate, where the conjugating map is given by $\hbar \colon x \mapsto 1-x$.  This is immediate from the definition of the maps involved.  Similarly, the maps  $\widetilde{T}_{\beta, \alpha}^{+}$ and $\widetilde{T}_{\beta, 2-\alpha-\beta}^{-}$ are topologically conjugate and the maps $\widetilde{T}_{\beta, \alpha}^{-}$ and $\widetilde{T}_{\beta, 2-\alpha-\beta}^{+}$ are topologically conjugate, where the conjugating map is also $\hbar$.  When lifted to $\{ 0, 1\}^{\mathbb{N}}$ the mapping $\hbar$ becomes the \textit{symmetric map} $* \colon \{ 0, 1\}^{\mathbb{N}} \circlearrowleft$ defined by,
\begin{align*}
*(\omega_{1}, \omega_{2}, \dots) = (\omega_{1} +1 \bmod 2, \omega_{2} +1 \bmod 2, \dots),
\end{align*}
for all $\omega \in \{ 0, 1\}^{\mathbb{N}}$.  From this we conclude that the following couples are topologically conjugate where the conjugating map is the symmetric map $*$:
\[
\sigma\lvert_{\Omega_{\beta,\alpha}^{-}} \; \text{and} \; \sigma\lvert_{\Omega_{\beta, 2-\alpha-\beta}^{+}},
\;\;\,
\sigma\lvert_{\Omega_{\beta,\alpha}^{+}} \; \text{and} \; \sigma\lvert_{\Omega_{\beta, 2-\alpha-\beta}^{-}},
\;\;\,
\sigma\lvert_{\widetilde{\Omega}_{\beta,\alpha}^{-}} \; \text{and} \; \sigma\lvert_{\widetilde{\Omega}_{\beta, 2-\alpha-\beta}^{+}}
\;\;\,
\text{and}
\;\;\,
\sigma\lvert_{\widetilde{\Omega}_{\beta,\alpha}^{+}} \; \text{and} \; \sigma\lvert_{\widetilde{\Omega}_{\beta, 2-\alpha-\beta}^{-}}.
\]

\begin{theorem}\label{thm:BM2011}
For all $(\beta, \alpha) \in \Delta$ we have that $\Omega^{\pm}_{\beta, \alpha} = \Omega^{\mp}_{\beta, 2 - \beta - \alpha}$ and $\widetilde{\Omega}^{\pm}_{\beta, \alpha} = \widetilde{\Omega}^{\mp}_{\beta, 2 - \beta - \alpha}$. Moreover, given $\beta \in (1, 2)$, there exists a unique point $\alpha \in [0, 2-\beta]$, namely $\alpha = 1 - \beta/2$, such that $\tau_{\beta, \alpha}^{+}(p) = *(\tau_{\beta, \alpha}^{-}(p))$, where $p = p_{\beta, \alpha} = (1-\alpha)/\beta$.
\end{theorem}

\begin{proof}
The first part of the result follows immediately from the above comments and Theorem~\ref{thm:Structure}.  The second part of the result follows from an application of \cite[Theorem 2]{BM:2011} together with the fact that the maps $\widetilde{T}_{\beta, \alpha}^{\pm}$ are topologically conjugate to a sub-class of maps considered in \cite{BM:2011}, where the conjugating map $\hbar_{\beta, \alpha} \colon [0, 1] \to [-\alpha/(\beta - 1), (1- \alpha)/(\beta - 1)]$ is given by $\hbar_{\beta, \alpha}(x) = (x - \alpha)/(\beta - 1)$.
\end{proof}

\section{Proofs of Theorem~\ref{thm:ESFTP} and Corollary~\ref{cor:fibersSFT}}\label{sec:TheoremA}

\noindent Let us first prove an analogous result for the extended model $\widetilde{\Omega}_{\beta, \alpha}$.  Throughout this section, for a given $(\beta, \alpha) \in \Delta$, we set $p = p_{\beta, \alpha}$.

\begin{proposition}\label{prop:ESFTP}
Let $\beta \in (1, 2)$.
\begin{enumerate}
\item\label{propA:1} If $\alpha \in (0, 2 - \beta)$, then $\widetilde{\Omega}_{\beta, \alpha}$ is a SFT if and only if both $\tau^{\pm}_{\beta, \alpha}(p)$ are periodic.
\item\label{propA:2} If $\alpha = 2 - \beta$, then $\widetilde{\Omega}_{\beta, \alpha}$ is a SFT if and only if $\tau^{-}_{\beta, \alpha}(p)$ is periodic.
\item\label{propA:3} If $\alpha = 0$, then $\widetilde{\Omega}_{\beta, \alpha}$ is a SFT if and only if $\tau^{+}_{\beta, \alpha}(p)$ is periodic.
\end{enumerate}
\end{proposition}

\begin{proof}[Proof of Proposition~\ref{prop:ESFTP}\ref{propA:1}]
We first prove that $\tau^{-}_{\beta, \alpha}(p)$ is periodic if $\widetilde{\Omega}_{\beta, \alpha}$ is a SFT.  Assume by way of contradiction that $\widetilde{\Omega}_{\beta, \alpha}$ is a SFT, but that $\tau^{-}_{\beta, \alpha}(p)$ is not periodic.  Then by Lemma~\ref{lem:keylemma} the infinite word $\nu = \sigma (\tau_{\beta, \alpha}^{-}(p))$ is zero-full, namely
\begin{align*}
N_{0}(\nu) = \left\lvert \left\{ k > 2 \colon (\nu_{1}, \nu_{2}, \dots, \nu_{k-1}, 1) \in \widetilde{\Omega}_{\beta, \alpha}^{-}\lvert_{k} \; \text{and} \; \nu_{k} = 0 \right\} \right\rvert
\end{align*}
is not finite.  Thus, there exists a non-constant monotonic sequence $(n_{k})_{k \in \mathbb{N}}$ of natural numbers such that
\begin{align*}
(\nu_{1}, \nu_{2}, \dots, \nu_{n_{k}-1} 1) \in \widetilde{\Omega}^{-}_{\beta, \alpha}\lvert^{*} \quad \text{and} \quad
(0, \nu_{1}, \nu_{2}, \dots, \nu_{n_{k}-1}, 0) = \tau_{\beta, \alpha}^{-}(p)\lvert_{n_{k}+1}.
\end{align*}
This together with Corollary~\ref{cor:shift_invarance} implies that $\widetilde{\Omega}_{\beta, \alpha}$ is not a SFT, contradicting our assumption.

The statement that if $\widetilde{\Omega}_{\beta, \alpha}$ is a SFT, then $\tau^{+}_{\beta, \alpha}(p)$ is periodic, follows in an identical manner to the above, where we use one-fullness instead of zero-fullness.

We will now prove the converse statement: for $(\beta, \alpha) \in \Delta$, if both the kneading invariants of $\widetilde{T}^{\pm}_{\beta, \alpha}$ are periodic, then $\widetilde{\Omega}_{\beta, \alpha}$ is a SFT.  Without loss of generality we may assume that both the kneading invariants of $\widetilde{T}^{\pm}_{\beta, \alpha}$ have the same period length $n$.  This together with Corollary~\ref{cor:shift_invarance} implies that it is sufficient to prove, for all $\omega = (\omega_{1}, \omega_{2}, \dots, \omega_{k})$, $\xi = (\xi_{1}, \xi_{2}, \dots, \xi_{m}) \in \widetilde{\Omega}_{\beta, \alpha}\lvert^{*}$ with $k > n$, that if
\begin{align}\label{eq:Periodic->SFT1}
(\omega_{k - n + 1}, \omega_{k - n + 2}, \dots, \omega_{k}, \xi_{1}, \xi_{2}, \dots, \xi_{m}) \in \widetilde{\Omega}_{\beta, \alpha}\lvert^{*},
\end{align}
then
\begin{align*}
(\omega_{1}, \omega_{2}, \dots, \omega_{k}, \xi_{1}, \xi_{2}, \dots, \xi_{m}) \in \widetilde{\Omega}_{\beta, \alpha}\lvert^{*}.
\end{align*}
If the inclusion in (\ref{eq:Periodic->SFT1}) holds, then there exists at least one $\eta = (\eta_{1}, \eta_{2}, \eta_{1} \dots) \in \{0, 1\}^{\mathbb{N}}$ with
\begin{align*}
(\omega_{k - n + 1}, \omega_{k - n + 2}, \dots, \omega_{k}, \xi_{1}, \xi_{2}, \dots, \xi_{m}, \eta_{1}, \eta_{2}, \dots) \in \widetilde{\Omega}_{\beta, \alpha}.
\end{align*}
Fix such an $\eta$ and consider the infinite word $\theta \coloneqq (\omega_{1}, \omega_{2}, \dots, \omega_{k}, \xi_{1}, \xi_{2}, \dots, \xi_{m}, \eta_{1}, \eta_{2}, \dots)$.  We claim that $\theta$ belongs to the space $\widetilde{\Omega}_{\beta, \alpha}$.  Suppose that this is not the case, then by Theorem~\ref{thm:Structure}, there exists $l \in \{ 0, 1, \dots, k- n -1\}$ such that
\begin{align*}
\tau_{\beta, \alpha}^{-}(p) \prec \sigma^{l}(\theta) \prec \tau_{\beta, \alpha}^{+}(p).
\end{align*}
By the definition of the lexicographic ordering and Corollary~\ref{cor:shift_invarance}, since both $\tau_{\beta,\alpha}^{\pm}(p)$ are periodic with period length $n$ and $\omega = (\omega_{1}, \omega_{2}, \dots, \omega_{k}) \in \widetilde{\Omega}_{\beta, \alpha}\lvert^{*}$, there exists $t \in \mathbb{N}$ with $l +(t -1)n < k \leq l + tn$ such that if $m < l + tn - k$, then
\begin{align*}
\tau_{\beta, \alpha}^{-}(p)\lvert_{n} \prec (\omega_{l+(t-1)n + 1}, \omega_{l+(t-1)n + 2}, \dots, \omega_{k}, \xi_{1}, \xi_{2}, \dots, \xi_{m}, \eta_{1}, \eta_{2}, \dots, \eta_{l + tn - k - m}) \prec \tau_{\beta, \alpha}^{+}(p)\lvert_{n},
\end{align*}
otherwise,
\begin{align*}
\tau_{\beta, \alpha}^{-}(p)\lvert_{n} \prec (\omega_{l+(t-1)n + 1}, \omega_{l+(t-1)n + 2}, \dots, \omega_{k}, \xi_{1}, \xi_{2}, \dots, \xi_{l + tn - k}) \prec \tau_{\beta, \alpha}^{+}(p)\lvert_{n}.
\end{align*}
Notice that $k \leq l+t n$ implies $k - n +1 \leq l + (t-1)n + 1$.  So
\begin{align*}
\tau_{\beta, \alpha}^{-}(p) \prec \sigma^{l + tn - k}(\omega_{k - n + 1}, \omega_{k - n + 2}, \dots, \omega_{k}, \xi_{1}, \xi_{2}, \dots, \xi_{m}, \eta_{1}, \eta_{2}, \dots) \prec \tau_{\beta, \alpha}^{+}(p),
\end{align*}
contradicting how $\eta$ was originally chosen.
\end{proof}

\begin{proof}[Proof of Proposition~\ref{prop:ESFTP}\ref{propA:2}]
The forward direction follows in the same way as in the proof of the forward direction of Proposition~\ref{prop:ESFTP}\ref{propA:1}.  Noting that, $p = 1 - 1/\beta$, since $\alpha = 2 - \beta$, and thus $\tau_{\beta, \alpha}^{+}(p) = (1, 0, 0, 0, \dots )$. The reverse direction follows analogously to the proof of the reverse direction of Proposition~\ref{prop:ESFTP}\ref{propA:1}, except that we set $n$ to be the length of the period of $\tau_{\beta, \alpha}^{-}(p)$.
\end{proof}

\begin{proof}[Proof of Proposition~\ref{prop:ESFTP}\ref{propA:3}]
This follows from Proposition~\ref{prop:ESFTP}\ref{propA:2} and Theorem~\ref{thm:BM2011}.
\end{proof}

We will shortly show in Proposition \ref{prop:SFText->SFT} that $\widetilde{\Omega}_{\beta, \alpha}$ is a SFT if and only if $\Omega_{\beta, \alpha}$ is a SFT.  This result together with Proposition~\ref{prop:ESFTP} completes the proof of Theorem~\ref{thm:ESFTP}.  In order to prove the backwards implication of Proposition~\ref{prop:SFText->SFT} we will use Theorem~\ref{thm:equivalent-def-SFT}.  Namely that a subshift $\Omega$ is a SFT if and only if there exists a finite set of forbidden word $F$ such that if $\xi \in F$, then $\sigma^{m}(\omega)\lvert_{\lvert \xi \rvert} \neq \xi$, for all $\omega \in \Omega$.

\begin{proposition}\label{prop:SFText->SFT}
For $(\beta, \alpha) \in \Delta$ the subshift $\widetilde{\Omega}_{\beta, \alpha}$ is a SFT if and only if the subshift $\Omega_{\beta, \alpha}$ is a SFT.
\end{proposition}

\begin{proof}
We first prove the forward direction: if the space $\widetilde{\Omega}_{\beta, \alpha}$ is a SFT, then $\Omega_{\beta, \alpha}$ is a SFT.  We proceed by way of contradiction.  Suppose that $\widetilde{\Omega}_{\beta, \alpha}$ is a SFT, but $\Omega_{\beta, \alpha}$ is not a SFT.   Under this assumption, Proposition~\ref{prop:ESFTP} implies that $\tau_{\beta, \alpha}^{\pm}(p)$ are both periodic and we may assume, without loss of generality, that they have the same period length $m$.  Moreover, by Corollary~\ref{cor:shift_invarance}, since, by assumption, $\Omega_{\beta, \alpha}$ is not a SFT, there exist $\omega = (\omega_{1}, \omega_{2}, \dots, \omega_{n}), \xi = (\xi_{1}, \xi_{2}, \dots, \xi_{k}) \in \Omega_{\beta, \alpha}\lvert^{*}$, with  $n, k \in \mathbb{N}$ and $n \geq m$ such that
\begin{align*}
(\omega_{n - m + 1}, \omega_{n - m + 2}, \dots, \omega_{n}, \xi_{1}, \xi_{2}, \dots, \xi_{k}) \in \Omega_{\beta, \alpha}\lvert^{*}
\; \text{but} \;
(\omega_{1}, \omega_{2}, \dots, \omega_{n}, \xi_{1}, \xi_{2}, \dots, \xi_{k}) \not\in \Omega_{\beta, \alpha}\lvert^{*}.
\end{align*}
In particular, as $\widetilde{\Omega}_{\beta, \alpha}$ is a SFT, Theorem~\ref{thm:Structure} implies that there exists an $l \in \{ 0, 1, \dots, n-1\}$ such that either
\begin{align*}
(\omega_{l + 1}, \omega_{l + 2}, \dots, \omega_{n}, \xi_{1}, \xi_{2}, \dots, \xi_{k}) \prec \tau_{\beta, \alpha}^{+}(0)\lvert_{m + n - l} = \sigma(\tau_{\beta, \alpha}^{+}(p))\lvert_{k + n - l}
\end{align*}
or
\begin{align*}
(\omega_{l + 1}, \omega_{l + 2}, \dots, \omega_{n}, \xi_{1}, \xi_{2}, \dots, \xi_{k}) \succ \tau_{\beta, \alpha}^{-}(1)\lvert_{m + n - l} = \sigma(\tau_{\beta, \alpha}^{-}(p))\lvert_{k + n - l}.
\end{align*}
However, since $\omega, \xi \in \Omega_{\beta, \alpha}\lvert^{*}$ and since $\tau_{\beta, \alpha}^{\pm}(p)$ are both periodic with period length $m$, it follows that $l + 1 \geq n - m$ contradicting the fact that $(\omega_{n - m + 1}, \dots, \omega_{n}, \xi_{1}, \dots, \xi_{k}) \in \Omega_{\beta, \alpha}\lvert^{*}$.

We will now show the converse: if the space $\Omega_{\beta, \alpha}$ is a SFT, then $\widetilde{\Omega}_{\beta, \alpha}$ is a SFT.   Let $F$ denote the set of forbidden words of $\Omega_{\beta, \alpha}$, and assume that $F$ is the smallest such set.  Then for each $\xi \in F$ either
\begin{align*}
\xi \prec \tau_{\beta, \alpha}^{+}(0)\lvert_{\lvert \xi \rvert}, \quad \xi \succ \tau_{\beta, \alpha}^{-}(1)\lvert_{\lvert \xi \rvert}, \quad \text{or} \quad  \tau_{\beta, \alpha}^{-}(p)\lvert_{\lvert \xi \rvert} \prec \xi \prec \tau_{\beta, \alpha}^{+}(p)\lvert_{\lvert \xi \rvert}.
\end{align*}
We claim that the set $\{ \xi \in F \colon \tau_{\beta, \alpha}^{-}(p)\lvert_{\lvert \xi \rvert} \prec \xi \prec \tau_{\beta, \alpha}^{+}(p)\lvert_{\lvert \xi \rvert} \}$ is a finite set of forbidden words for $\widetilde{\Omega}_{\beta,\alpha}$.  If this is not the case, then we obtain a contradiction to Theorem~\ref{thm:Structure}.
\end{proof}

\begin{proof}[Proof of Theorem~\ref{thm:ESFTP}]
This is an immediate consequence of Propositions~\ref{prop:ESFTP} and \ref{prop:SFText->SFT}.
\end{proof}

Let us conclude this section with the proof of Corollary~\ref{cor:fibersSFT}.

\begin{proof}[Proof of Corollary~\ref{cor:fibersSFT}]
By construction $\pi_{\beta, \alpha} \circ \tau_{\beta, \alpha}^{\pm}$ is the identity function on $[0,1]$ and so $\beta$ is a root of the polynomial
\[
\pi_{\beta, \alpha}(\tau^{+}_{\beta, \alpha}(p)) - \pi_{\beta, \alpha}(\tau^{-}_{\beta, \alpha}(p)) -\pi_{\beta, 0}(\tau^{+}_{\beta, \alpha}(p)) + \pi_{\beta, 0}(\tau^{-}_{\beta, \alpha}(p)).
\]
By the definition of the projection map $\pi_{\beta, \alpha}$, this latter polynomial is independent of $\alpha$ and all of its coefficients belong to the set $\{-1, 0, 1\}$.  An application of Theorem~\ref{thm:Parry1960} and \ref{thm:ESFTP} completes the proof.
\end{proof}

\section{Proof of Theorem~\ref{thm:SFTandTransitive}}\label{sec:TheoremB}

\noindent Having proved Theorem~\ref{thm:ESFTP}, we are now equipped to prove Theorem~\ref{thm:SFTandTransitive}. Before setting out to prove  Theorem~\ref{thm:SFTandTransitive}, let us recall the result of Palmer and Glendinning on the classification of the point $(\beta, \alpha) \in \Delta$ such that $T_{\beta,\alpha}^{+}$, and hence $T_{\beta,\alpha}^{-}$, is transitive.

\begin{definition}
Suppose that $1 < k < n$ are natural numbers with $\mathrm{gcd}(k, n) = 1$.  Let $s, m \in \mathbb{N}$ be such that $0 \leq s < k$ and $n = mk + s$.  For $1 \leq j \leq s$ define $V_{j}$ and $r_{j}$ by $jk = V_{j} s + r_{j}$, where $r_{j}, V_{j} \in \mathbb{N}$ and $0 \leq r_{j} < s$.  Define $h_{j}$ inductively via the formula $V_{j} = h_{1} + h_{2} + \dots + h_{j}$.  We let $D_{k, n}$ denote the set of points $(\beta, \alpha) \in \Delta$, such that $1 < \beta^{n} \leq 2$ and
\[
\frac{1 + \beta (\sum_{j=1}^{s} W_{j} - 1) }{\beta (\beta^{n-1} + \dots + 1)} \leq \alpha \leq \frac{\beta(\sum_{j=1}^{s} W_{j}) -\beta^{n+1} + \beta^{n} + \beta - 1}{\beta(\beta^{n-1} + \dots + 1)}.
\]
Here,
\[
W_{1} \coloneqq \sum_{i = 1}^{V_{1}} \beta^{(V_{s} - i)m + s - 1}
\quad \text{and} \quad
W_{j} \coloneqq \sum_{i = 1}^{h_{j}} \beta^{(V_{s} - V_{j-1} - i)m + s - j},
\]
for $2 \leq j \leq s$.  Further, for each natural number $n > 1$, we define the set $D_{1, n}$ to be the set of points $(\beta, \alpha) \in \Delta$ such that $1 < \beta^{n} \leq 2$ and
\[
\frac{1}{\beta(\beta^{n-1} + \dots + 1)} \leq \alpha \leq \frac{-\beta^{n+1} + \beta^{n} + 2\beta - 1}{\beta(\beta^{n-1} + \dots + 1)}.
\]
(See Figure~\ref{Fig:ParPlot} for an illustration of the regions $D_{k, n}$.)
\end{definition}

\begin{figure}[htbp]
\begin{center}
\scalebox{0.65}{
\includegraphics{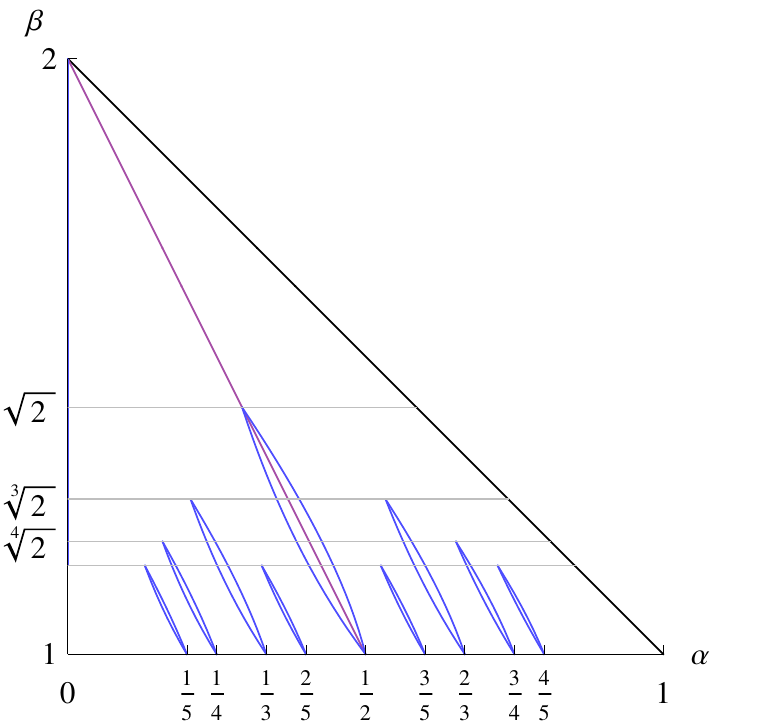}	
	}
\end{center}
\caption{Plot of the parameter space $\Delta$, the line $\alpha = 1 - \beta/2$ and the boundary of the regions $D_{1,2}$, $D_{1,3}$, $D_{2,3}$, $D_{1,5}$, $D_{2,5}$, $D_{3,5}$, $D_{4,5}$.}
\label{Fig:ParPlot}
\end{figure}

\begin{theorem}[{\cite{P:1979} and \cite[Proposition 2]{G:1990}}]\label{thm:Palmer}
If $(\beta,\alpha) \in D_{k, n}$, for some $1\leq k < n$ with $\mathrm{gcd}(k, n) = 1$, then $T_{\beta, \alpha}^{-}$, and hence $T_{\beta, \alpha}^{+}$, is not transitive.
\end{theorem}

We now present the proof of Theorem~\ref{thm:SFTandTransitive}.

\begin{proof}[Proof of Theorem~\ref{thm:SFTandTransitive}\ref{item:1:Density2}]
Fix $n, k \in \mathbb{N}$ with $n \geq 2$ and recall that $\alpha_{n, k} = (2 - \beta_{n, k})/2$.  Observe that
\begin{equation}\label{eq:1.lower1}
\frac{2 - \beta_{n, k}}{2} \geq \frac{1}{\beta_{n, k}(\beta_{n, k} + 1)}
\quad \text{if and only if} \quad
(2 - \beta^{2}_{n, k})(\beta_{n, k} - 1) \geq 0
\end{equation}
and that
\begin{equation}\label{eq:2.upper2}
\frac{2 - \beta_{n, k}}{2} \leq \frac{-\beta_{n, k}^{3} + \beta_{n, k}^{2} + 2\beta_{n, k} - 1}{\beta_{n, k}(\beta_{n, k} + 1)}
\quad \text{if and only if} \quad (\beta_{n, k}^{2} - 2)(\beta_{n, k} - 1) \leq 0.
\end{equation}
The later inequalities in \eqref{eq:1.lower1} and \eqref{eq:2.upper2} both hold true since $
1 < \beta_{n, k} < (\beta_{n, k})^{2} \leq (\beta_{n, k})^{2^{k}} = \gamma_{n}$ and since $\gamma_{n} \in (1, 2)$.  Therefore, $(\beta_{n, k}, \alpha_{n, k}) \in D_{1, 2}$ and the results follows from an application of Theorem \ref{thm:Palmer}.
\end{proof}

In order to prove Theorem~\ref{thm:SFTandTransitive}\ref{item:2:Density2} we will use the following.  For $n \in \mathbb{N}$ set
\[
\xi_{n, 0}^{-} = (\overline{0, \underbrace{1, 1, \dots, 1}_{n -\text{times}}})
\]
and, for $k \geq 1$, set
\[
\xi_{n, k}^{-} = (\overline{ \xi_{n, k-1}^{-}\lvert_{2^{k-1}}, *(\xi_{n, k-1}^{-}\lvert_{2^{k-1}}), \underbrace{\underbrace{*(\xi_{n, k-1}^{-})\lvert_{2^{k-1}}, \xi_{n, k-1}^{-}\lvert_{2^{k-1}}},\dots, \underbrace{*(\xi_{n, k-1}^{-})\lvert_{2^{k-1}},\xi_{n, k-1}^{-}\lvert_{2^{k-1}}}}_{n-\text{times}}}).
\]
Further, for $n, k \in \mathbb{N}$, we set $\xi_{n, k}^{+} \coloneqq *(\xi_{n, k}^{-})$.  Note that, $\xi_{n, k}^{-}\lvert_{2^{k}} \prec \xi_{n, k}^{+}\lvert_{2^{k}}$, for all $n, k \in \mathbb{N}$.  Observe that if $k \geq 2$ and if $l \in \{ 2, 3, \dots, k \}$, then
\begin{equation}\label{eq:fourblock}
\begin{split}
\xi_{n,k}^{-}\lvert_{2^{k-l+2}} &= \left(\xi_{n,k - l}^{-}\lvert_{2^{k-l}}, *(\xi_{n,k - l}^{-}\lvert_{2^{k-l}}), *(\xi_{n,k - l}^{-}\lvert_{2^{k-l}}), \xi_{n,k - l}^{-}\lvert_{2^{k-l}}\right)\\
\xi_{n,k}^{+}\lvert_{2^{k-l+2}} &= \left(*(\xi_{n,k - l}^{-}\lvert_{2^{k-l}}), \xi_{n,k - l}^{-}\lvert_{2^{k-l}}, \xi_{n,k - l}^{-}\lvert_{2^{k-l}}, *(\xi_{n,k - l}^{-}\lvert_{2^{k-l}})\right).
\end{split}
\end{equation}
Using the above, Theorems~\ref{thm:ESFTP} and \ref{thm:BM2011} together with the results of \cite{BV:2012},  and the following three propositions (Propositions~\ref{prop:(1)}, \ref{prop:(2)}, \ref{prop:(3)}), we obtain Theorem~\ref{thm:SFTandTransitive}\ref{item:2:Density2}.

\begin{proof}[Proof of Theorem~\ref{thm:SFTandTransitive}\ref{item:2:Density2}]
Let $n, k \in \mathbb{N}_{0}$ with $n \geq 2$ be fixed.  By combining Theorems~\ref{thm:ESFTP} and \ref{thm:BM2011} together \cite[Theorem 1]{BV:2012} it suffices to show that for all $n, k \in \mathbb{N}_{0}$ with $n \geq 2$, we have the following assertions.
\begin{enumerate}[label*=(\alph*)]
\item\label{i-nk} The pair $(\xi_{n, k}^{-}, \xi_{n, k}^{+})$ is \textit{self-admissible}, namely that, for all $m \in \mathbb{N}$,
\[
\left( \sigma^{m}(\xi_{n, k}^{-}) \preceq \xi_{n, k}^{-} \; \text{or} \; \sigma^{m}(\xi_{n, k}^{-}) \succ \xi_{n, k}^{+}\right)
\;\; \text{and} \;\;
\left(\sigma^{m}(\xi_{n, k}^{+}) \prec \xi_{n, k}^{-} \; \text{or} \; \sigma^{m}(\xi_{n, k}^{+}) \succeq \xi_{n, k}^{+}\right).
\]
(This is precisely Condition (1) of lex-admissible as defined in \cite{BV:2012}.)
\item\label{ii-nk} The value $\beta_{n, k}$ is the maximal real root of the polynomial
\[
x \mapsto \sum_{m = 1}^{\infty} (\xi_{n,k, m}^{-} - \xi_{n,k, m}^{+}) x^{-m},
\]
where $\xi_{n,k, m}^{\pm}$ denotes the $m$-th letter of the word $\xi_{n, k}^{\pm}$, respectively.  (By \cite[Lemma 3]{BV:2012}, this is Condition (2) of lex-admissible as defined in \cite{BV:2012}.)
\item\label{iii-nk} The following equalities hold:
\[
\pi_{\beta_{n,k}, \alpha_{n,k}}(\xi_{n, k}^{-}) = 1/2 = \pi_{\beta_{n,k}, \alpha_{n,k}}(\xi_{n, k}^{+}).
\]
(By \cite[Lemma 7]{BV:2012}, this is Condition (3) of lex-admissible as given in \cite{BV:2012}.)
\end{enumerate}
However, \ref{i-nk} is precisely the result given in Proposition~\ref{prop:(1)}, \ref{ii-nk} is the result presented in Proposition~\ref{prop:(2)} and \ref{iii-nk} is the result given in Proposition~\ref{prop:(3)}; these propositions, together with their proofs, are presented below.
\end{proof}

\begin{proposition}\label{prop:(1)}
For $n, k \in \mathbb{N}_{0}$ with $n \geq 2$, we have that the pair $(\xi_{n, k}^{-}, \xi_{n, k}^{+})$ is self-admissible.
\end{proposition}

\begin{proof}
Using the observations that
\[
\begin{aligned}
\xi_{n, 0}^{-} &= (\overline{0, \underbrace{1, 1, \dots, 1}_{n -\text{times}}}),
&\xi_{n, 0}^{+} &= (\overline{1, \underbrace{0, 0, \dots, 0}_{n -\text{times}}}),\\
\xi_{n, 1}^{-} &= (\overline{ 0, 1, \hspace{-1.5mm}\underbrace{\underbrace{1 , 0}\hspace{-1.25mm},\dots, \hspace{-1.5mm} \underbrace{1, 0}}_{n-\text{times}}}\hspace{-1mm}),
&\xi_{n, 1}^{+} &= (\overline{ 1, 0, \hspace{-1.5mm}\underbrace{\underbrace{0 , 1}\hspace{-1.25mm},\dots, \hspace{-1.5mm}\underbrace{0, 1}}_{n-\text{times}}}\hspace{-1mm}),\\
\xi_{n, 2}^{-} &= (\overline{ 0, 1, 1, 0, \underbrace{\underbrace{1 , 0, 0, 1},\dots, \underbrace{1, 0, 0, 1}}_{n-\text{times}}}), \qquad
&\xi_{n, 2}^{+} &= (\overline{ 1, 0, 0, 1, \underbrace{\underbrace{0 , 1, 1, 0},\dots, \underbrace{0, 1, 1, 0}}_{n-\text{times}}}),
\end{aligned}
\]
it is easy to verify, using the repeating structure of the words $\xi_{n, k}^{\pm}$, that the pairs $(\xi_{n, 0}^{-}, \xi_{n, 0}^{+})$, $(\xi_{n, 1}^{-}, \xi_{n, 1}^{+})$ and $(\xi_{n, 2}^{-}, \xi_{n, 2}^{+})$ are self-admissible.

We now proceed by induction on $k$.  Assume that $k > 2$ and let $a$ denote the finite word $\xi_{n, 2}^{-}\lvert_{4} = (0, 1, 1, 0)$.  It is elementary to show that, for $m \in \{ 1, 2, 3\}$,
\begin{enumerate}[label*=(\alph*)]
\item $\sigma^{m}(a, a)\prec (a, *(a))\lvert_{8 - m}$ or $\sigma^{m}(a, a)\succ (*(a), a)\lvert_{8 - m}$,
\item $\sigma^{m}(a,*(a))\prec (a, *(a))\lvert_{8 - m}$ or $\sigma^{m}(a, *(a))\succ (*(a), a)\lvert_{8 - m}$,
\item $\sigma^{m}(*(a), *(a))\prec (a, *(a))\lvert_{8 - m}$ or  $\sigma^{m}(*(a), *(a))\succ (*(a), a)\lvert_{8 - m}$ and
\item $\sigma^{m}(*(a), a)\prec (a, *(a))\lvert_{8 - m} $ or $\sigma^{m}(*(a), a)\succ (*(a), a)\lvert_{8 - m}$,
\end{enumerate}
By construction, for $k > 2$, the sequence $\xi_{n,k}^{-}$ is an infinite concatenation of the finite words $a$ and $*(a)$ and, by \eqref{eq:fourblock}, we have that $\xi_{n, k}^{-}\lvert_{8} = (a, *(a))$.  Thus, for any $m = 4p + l$, where $p \in \mathbb{N}_{0}$ and $l \in \{ 1, 2, 3\}$, we have that
\begin{equation*}
\left(\sigma^{m}(\xi_{n, k}^{-}) \preceq \xi_{n, k}^{-} \quad \text{or} \quad \sigma^{m}(\xi_{n, k}^{-}) \succ \xi_{n, k}^{+} \right)
\;\; \text{and} \;\;
\left(\sigma^{m}(\xi_{n, k}^{+}) \preceq \xi_{n, k}^{-} \quad \text{or} \quad \sigma^{m}(\xi_{n, k}^{+}) \succ \xi_{n, k}^{+}\right).
\end{equation*}
Hence, all that remains is to prove, for $k > 2$ and for $m = 4 p$, where $p \in \mathbb{N}$, that
\[
\left(\sigma^{m}(\xi_{n, k}^{-}) \preceq \xi_{n, k}^{-} \quad \text{or} \quad \sigma^{m}(\xi_{n, k}^{-}) \succ \xi_{n, k}^{+}\right)
\quad \text{and} \quad
\left( \sigma^{m}(\xi_{n, k}^{+}) \prec \xi_{n, k}^{-} \quad \text{or} \quad \sigma^{m}(\xi_{n, k}^{+}) \succeq \xi_{n, k}^{+} \right).
\]
Indeed, for $p \in \mathbb{N}$ and $k > 2$,
\[
\begin{split}
&\sigma^{2^{k}p}(\xi_{n,k}^{-})\\
&= \sigma^{2^{k}p}(\overline{\xi_{n,k}^{-}\lvert_{2^{k}}, \underbrace{*(\xi_{n,k}^{-}\lvert_{2^{k}}), \dots, *(\xi_{n,k}^{-}\lvert_{2^{k}})}_{n-\text{times}}})\\
&= \begin{cases}
(\overline{\xi_{n,k}^{-}\lvert_{2^{k}}, \underbrace{*(\xi_{n,k}^{-}\lvert_{2^{k}}), \dots, *(\xi_{n,k}^{-}\lvert_{2^{k}})}_{n-\text{times}}})
& \text{if} \; p = (n+1) m \; \text{and} \; m \in \mathbb{N},\\
\\
(\underbrace{*(\xi_{n,k}^{-}\lvert_{2^{k}}), \dots, *(\xi_{n,k}^{-}\lvert_{2^{k}})}_{((n+1)(m+1)-p)-\text{times}}, \overline{\xi_{n,k}^{-}\lvert_{2^{k}}, \underbrace{*(\xi_{n,k}^{-}\lvert_{2^{k}}), \dots *(\xi_{n,k}^{-}\lvert_{2^{k}}}_{n-\text{times}}}) )
& \text{\parbox{11em}{if $p \in \{(n+1) m +1, \dots,$ $(n+1)(m+1) -1\}$ and $m \in \mathbb{N}$.}}
\end{cases}
\end{split}
\]
Thus by the equality given in \eqref{eq:fourblock} it follows that
\[
\sigma^{2^{k}p}(\xi_{n,k}^{-}) \begin{cases}
= \xi_{n, k}^{-} & \text{if} \; p = (n + 1)m \; \text{for some} \; m \in \mathbb{N},\\
\succ \xi_{n, k}^{+} & \text{otherwise.}
\end{cases}
\]
By a symmetrical argument we obtain that
\[
\sigma^{2^{k}p}(\xi_{n,k}^{+}) \begin{cases}
= \xi_{n, k}^{+} & \text{if} \; p = (n + 1)m \; \text{for some} \; m \in \mathbb{N},\\
\prec \xi_{n, k}^{-} & \text{otherwise.}
\end{cases}
\]
Now assume that for some $l \in \{1, 2, \dots, k - 2\}$ and for all $p \in \mathbb{N}$ that
\[
\left(\sigma^{2^{k-(l-1)} p}(\xi_{n, k}^{-}) \preceq \xi_{n, k}^{-} \; \text{or} \; \sigma^{2^{k-(l-1)} p}(\xi_{n, k}^{-}) \succ \xi_{n, k}^{+}\right)
\]
and that
\[
\left(\sigma^{2^{k-(l-1)} p}(\xi_{n, k}^{+}) \prec \xi_{n, k}^{-} \; \text{or} \; \sigma^{2^{k-(l-1)} p}(\xi_{n, k}^{+}) \succeq \xi_{n, k}^{+}\right).
\]
To complete the proof we will show, for all odd $p \in \mathbb{N}$, that
\[
\left(\sigma^{2^{k-l} p}(\xi_{n, k}^{-}) \preceq \xi_{n, k}^{-} \;\;\; \text{or} \;\;\; \sigma^{2^{k-l} p}(\xi_{n, k}^{-}) \succ \xi_{n, k}^{+} \right)
\]
and
\[
\left( \sigma^{2^{k-l} p}(\xi_{n, k}^{+}) \prec \xi_{n, k}^{-} \;\;\; \text{or} \;\;\; \sigma^{2^{k-l} p}(\xi_{n, k}^{+}) \succeq \xi_{n, k}^{+}\right).
\]
To this end observe that by construction, the words $\xi_{n, k}^{-}$ and $\xi_{n, k}^{+}$ are made up of the finite words $\xi_{n, k-(l-1)}^{-}\lvert_{2^{k-(l-1)}}$ and $*(\xi_{n, k-(l-1)}^{-}\lvert_{2^{k-(l-1)}})$ and moreover, that
\[
\xi_{n, k-(l-1)}^{-}\lvert_{2^{k-(l-1)}} = (\xi_{n, k - l}^{-}\lvert_{2^{k - l}}, *(\xi_{n, k - l}^{-}\lvert_{2^{k - l}})).
\]
Hence, we have one the following cases.
\begin{enumerate}[label*=(\alph*)]\setcounter{enumi}{4}
\item\label{case(a)A*(A)} $\sigma^{2^{k-l} p}(\xi_{n, k}^{-})\lvert_{3 \cdot 2^{k-l}} = (*(\xi_{n, k-l}^{-}\lvert_{2^{k-l}}), *(\xi_{n, k-l}^{-}\lvert_{2^{k-l}}), \xi_{n, k-l}^{-}\lvert_{2^{k-l}})$\\
$\sigma^{2^{k-l} p}(\xi_{n, k}^{+})\lvert_{3 \cdot 2^{k-l}} = (\xi_{n, k-l}^{-}\lvert_{2^{k-l}}, \xi_{n, k-l}^{-}\lvert_{2^{k-l}}, *(\xi_{n, k-l}^{-}\lvert_{2^{k-l}}))$
\item\label{case(b)AA} $\sigma^{2^{k-l} p}(\xi_{n, k}^{-})\lvert_{3 \cdot 2^{k-l}} = (*(\xi_{n, k-l}^{-}\lvert_{2^{k-l}}),\xi_{n, k-l}^{-}\lvert_{2^{k-l}}), *(\xi_{n, k-l}^{-}\lvert_{2^{k-l}}))$\\ 
$\sigma^{2^{k-l} p}(\xi_{n, k}^{+})\lvert_{3 \cdot 2^{k-l}} = (\xi_{n, k-l}^{-}\lvert_{2^{k-l}},*(\xi_{n, k-l}^{-}\lvert_{2^{k-l}}), \xi_{n, k-l}^{-}\lvert_{2^{k-l}})$
\item\label{case(c)*(A)A} $\sigma^{2^{k-l} p}(\xi_{n, k}^{-})\lvert_{3 \cdot 2^{k-l}} = (\xi_{n, k-l}^{-}\lvert_{2^{k-l}}, \xi_{n, k-l}^{-}\lvert_{2^{k-l}}, *(\xi_{n, k-l}^{-}\lvert_{2^{k-l}}))$\\ 
$\sigma^{2^{k-l} p}(\xi_{n, k}^{+})\lvert_{3 \cdot 2^{k-l}} = (*(\xi_{n, k-l}^{-}\lvert_{2^{k-l}}), *(\xi_{n, k-l}^{-}\lvert_{2^{k-l}}), \xi_{n, k-l}^{-}\lvert_{2^{k-l}})$
\item\label{case(d)*(A)*(A)} $\sigma^{2^{k-l} p}(\xi_{n, k}^{-})\lvert_{3 \cdot 2^{k-l}} = (\xi_{n, k-l}^{-}\lvert_{2^{k-l}}, *(\xi_{n, k-l}^{-}\lvert_{2^{k-l}}), \xi_{n, k-l}^{-}\lvert_{2^{k-l}})$\\ 
$\sigma^{2^{k-l} p}(\xi_{n, k}^{+})\lvert_{3 \cdot 2^{k-l}} = (*(\xi_{n, k-l}^{-}\lvert_{2^{k-l}}), \xi_{n, k-l}^{-}\lvert_{2^{k-l}}, *(\xi_{n, k-l}^{-}\lvert_{2^{k-l}}))$
\end{enumerate}
By the equality given in \eqref{eq:fourblock}, if \ref{case(a)A*(A)} or \ref{case(b)AA} occurs, then $\sigma^{2^{k-l} p}(\xi_{n, k}^{-}) \succ \xi_{n, k}^{+}$ and $\sigma^{2^{k-l} p}(\xi_{n, k}^{+}) \prec \xi_{n, k}^{-}$; otherwise, if \ref{case(c)*(A)A} or \ref{case(d)*(A)*(A)} occurs, then $\sigma^{2^{k-l} p}(\xi_{n, k}^{-}) \prec \xi_{n, k}^{-}$ and $\sigma^{2^{k-l} p}(\xi_{n, k}^{+}) \succ \xi_{n, k}^{-}$.  This completes the proof.
\end{proof}

\begin{proposition}\label{prop:(2)}
For $n, k \in \mathbb{N}_{0}$ with $n \geq 2$, the value $\beta_{n, k}$ is the maximal real root of the polynomial
\[
x \mapsto \sum_{m = 1}^{\infty} (\xi_{n,k, m}^{-} - \xi_{n,k, m}^{+}) x^{-m}.
\]
\end{proposition}

\begin{proof}
Observe, by the recursive definition of the infinite words $\xi_{n,k}^{\pm}$, that
\begin{align*}
\sum_{m = 1}^{\infty} (\xi_{n,k, m}^{-} - \xi_{n,k, m}^{+}) x^{-m}
&= \frac{x^{(n+1) 2^{k}}}{x^{(n+1) 2^{k}} - 1} \sum_{m = 1}^{(n+1) 2^{k}} (\xi_{n,k, m}^{-} - \xi_{n,k, m}^{+}) x^{-m}\\
&= \frac{x^{2^{k}} ( x^{n 2^{k}} - x^{(n-1) 2^{k}} - \dots - x^{2^{k}} - 1 )}{x^{(n+1) 2^{k}} - 1} \sum_{m = 1}^{2^{k}} (\xi_{n,k, m}^{-} - *(\xi_{n,k, m}^{-})) x^{-m}.
\end{align*}
This latter term is equal to zero if and only if either
\begin{align*}
x^{2^{k}} = 0,
\quad
x^{n 2^{k}} - x^{(n-1) 2^{k}} - \dots - x^{2^{k}} - 1 = 0
\quad \text{or} \quad
\sum_{m = 1}^{2^{k}} (\xi_{n,k, m}^{-} - *(\xi_{n,k, m}^{-})) x^{-m} = 0.
\end{align*}
However, $x^{2^{k}} > 0$ for all $x > 1$, and, by the definition of a multinacci number, the maximal real solution of $x^{n 2^{k}} - x^{(n-1) 2^{k}} - \dots - x^{2^{k}} - 1 = 0$ is $x = \beta_{n, k}$.  Further, we claim that
\[
\sum_{m = 1}^{2^{k}} (\xi_{n,k, m}^{-} - *(\xi_{n,k, m}^{-})) x^{-m} < 0,
\]
for all $x > 1$.  Given this claim, which we will shortly prove, the result follows.

To prove the claim we proceed by induction on $k$.  If $k = 0$, then
\[
\sum_{m = 1}^{2^{k}} (\xi_{n,k, m}^{-} - *(\xi_{n,k, m}^{-})) x^{-m} = -x^{-1}.
\]
Suppose that the statement is true for some $k \in \mathbb{N}_{0}$, then
\begin{align*}
\sum_{m = 1}^{2^{k+1}} (\xi_{n,k+1, m}^{-} - *(\xi_{n,k+1, m}^{-})) x^{-m} &=
\sum_{m = 1}^{2^{k}} (\xi_{n,k, m}^{-} - *(\xi_{n,k, m}^{-})) x^{-m} +
x^{-2^{k}} \sum_{m = 1}^{2^{k}} (*(\xi_{n,k, m}^{-}) - \xi_{n,k, m}^{-}) x^{-m}\\
&= \left( 1 - x^{-2^{k}} \right) \sum_{m = 1}^{2^{k}} (\xi_{n,k, m}^{-} - *(\xi_{n,k, m}^{-})) x^{-m}.
\end{align*}
Since  $x > 1$ the value of the term $1 - x^{-2^{k}}$ is positive and finite and by our inductive hypothesis
\[
\sum_{m = 1}^{2^{k}} (\xi_{n,k, m}^{-} - *(\xi_{n,k, m}^{-})) x^{-m} < 0.
\]
This completes the proof of the claim.
\end{proof}

\begin{proposition}\label{prop:(3)}
For all integers $n, k \in \mathbb{N}$ with $n \geq 2$, we have
\[
\pi_{\beta_{n,k}, \alpha_{n,k}}(\xi_{n, k}^{-}) = 1/2 = \pi_{\beta_{n,k}, \alpha_{n,k}}(\xi_{n, k}^{+}).
\]
\end{proposition}

This proposition yields that, if there exists an intermediate $\beta$-transformation $T^{\pm}$ with intermediate $\beta$-shift determined by $\xi_{n, k}^{\pm}$, then the points of discontinuity of $T^{\pm}$ are at $1/2$.

\begin{proof}[Proof of Proposition~\ref{prop:(3)}]
For the proof of this result we require the following.  Define $\kappa \colon \{ 0, 1\} \to \{ (0, 1), (1, 0) \}$ by $\kappa(0) \coloneqq (0, 1)$ and $\kappa(1) \coloneqq (1, 0)$.  Let $\widetilde{\kappa}$ be the substitution map defined on the set of all finite words in the alphabet $\{ 0, 1 \}$ by
\[
\widetilde{\kappa}(\omega_{1}, \omega_{2}, \dots, \omega_{m} ) \coloneqq (\kappa(\omega_{1}), \kappa(\omega_{2}), \dots, \kappa(\omega_{m})).
\]
We claim that 
\begin{equation}\label{eq:kappa}
\widetilde{\kappa}(\xi_{n, k-1}^{\pm}\lvert_{2^{k-1}}) = \xi_{n, k}^{\pm}\lvert_{2^{k}},
\end{equation}
for all $n, k \in \mathbb{N}$.  We will shortly prove the claim, then using \eqref{eq:kappa} we prove via an inductive argument the statement of the proposition.

Fix $n \in \mathbb{N}$.  We now prove \eqref{eq:kappa} by induction on $k$.  By definition we have
\[
\begin{aligned}
\xi_{n, 0}^{-} &= (\overline{0, \underbrace{1, 1, \dots, 1}_{n -\text{times}}}),
&\xi_{n, 0}^{+} &= (\overline{1, \underbrace{0, 0, \dots, 0}_{n -\text{times}}}),\\
\xi_{n, 1}^{-} &= (\overline{ 0, 1, \hspace{-1.5mm}\underbrace{\underbrace{1, 0}\hspace{-1.5mm},\dots, \hspace{-1.5mm}\underbrace{1, 0}}_{n-\text{times}}}\hspace{-0.75mm}),
&\xi_{n, 1}^{+} &= (\overline{ 1, 0, \hspace{-1.5mm}\underbrace{\underbrace{0 , 1}\hspace{-1.5mm},\dots, \hspace{-1.5mm} \underbrace{0, 1}}_{n-\text{times}}}\hspace{-0.75mm}).
\end{aligned}
\]
This completes the proof of the base case of the induction.  So suppose that \eqref{eq:kappa} holds true for all integers $m < k$, for some integer $k > 1$.  Then, by the definition of $\xi_{n, k}$, the definition of $\widetilde{\kappa}$ and the inductive hypothesis, we have that
\[
\begin{aligned}
\widetilde{\kappa}(\xi_{n, k-1}^{\pm}\lvert_{2^{k-1}})
	&= \widetilde{\kappa}(\xi_{n, k-2}^{\pm}\lvert_{2^{k-2}}, *(\xi_{n, k-2}^{\pm}\lvert_{2^{k-2}}))\\
	&= (\widetilde{\kappa}(\xi_{n, k-2}^{\pm}\lvert_{2^{k-2}}), \widetilde{\kappa}(*(\xi_{n, k-2}^{\pm}\lvert_{2^{k-2}})))\\
	&= ((\xi_{n,k-1}^{\pm}\lvert_{2^{k-1}}, *(\xi_{n,k-1}^{\pm}\lvert_{2^{k-1}}) )\lvert_{2^{k-1}}, (*(\xi_{n,k-1}^{\pm}\lvert_{2^{k-1}})\lvert_{2^{k-1}}, \xi_{n,k-1}^{\pm}\lvert_{2^{k-1}})\lvert_{2^{k-1}})\\
	&= (\xi_{n,k-1}^{\pm}\lvert_{2^{k-1}}, *(\xi_{n,k-1}^{\pm}\lvert_{2^{k-1}}))\\
	&= \xi_{n,k}^{\pm}\lvert_{2^{k}}.
\end{aligned}
\]
This completes the induction and hence the proof of the claim.

For a fixed integer $n \geq 2$, we now prove, via an inductive argument on $k$, the statement of the proposition.  Recall that $\alpha_{n, k} = 1 - \beta_{n, k}/2$ and observe, for $k = 0$, by the definition of the projection map $\pi_{\beta_{n, 0}, \alpha_{n, 0}}$, that
\[
\pi_{\beta_{n, 0}, \alpha_{n, 0}}(\xi_{n, 0}^{+})
	= \frac{1 - \beta_{n,0}/2}{1 - \beta_{n, 0}} + \frac{1}{\beta_{n, 0}} \sum_{m = 0}^{\infty} \frac{1}{{\beta_{n,0}}^{(n+1) m}}
	= \frac{1 - \beta_{n, 0}/2}{1 - \beta_{n, 0}} + \frac{{\beta_{n, 0}}^{n}}{{\beta_{n, 0}}^{n+1} - 1}
	= \frac{1}{2}.
\]
The last equality follows by using the fact that $\beta_{n, 0}$ is the multinacci number of order $n$ and so $1 = {\beta_{n, 0}}^{-1} + {\beta_{n, 0}}^{-2} + \dots + {\beta_{n, 0}}^{-n}$.  Further, we note that
\[
\pi_{\beta_{n, 0}, \alpha_{n, 0}}(\xi_{n, 0}^{-})
	= \frac{1 - \beta_{n,0}/2}{1 - \beta_{n, 0}} + \frac{1}{\beta_{n,0}}\left(\frac{1}{{\beta_{n,0}}^{1}}  + \dots + \frac{1}{{\beta_{n,0}}^{n}}\right) \sum_{m = 1}^{\infty} \frac{1}{{\beta_{n,0}}^{(n+1) m}}
	= \pi_{\beta_{n, 0}, \alpha_{n, 0}}(\xi_{n, 0}^{+}).
\]
This completes the proof of the base case of the induction.  So suppose the statement of the proposition holds true for all integers $m < k$, for some $k \in \mathbb{N}$.  By the definition of the projection map $\pi_{\beta_{n, k}, \alpha_{n, k}}$, the fact that ${\beta_{n, k}}^{2} = \beta_{n, k-1}$ and \eqref{eq:kappa}, we have
\[
\begin{aligned}
\pi_{\beta_{n, k}, \alpha_{n, k}}(\xi_{n, k}^{\pm})
&= \beta_{n, k} \pi_{\beta_{n, k-1}}(\xi_{n, k-1}^{\pm}) + \pi_{\beta_{n, k-1}}(\xi_{n, k-1}^{\mp}) - \frac{(\beta_{n, k} + 1)(1 - {\beta_{n, k}}^{2}/2)}{1 - {\beta_{n, k}}^{2}} + \frac{1 - \beta_{n, k}/2}{1 - \beta_{n, k}}\\
&= \frac{\beta_{n,k} + 1}{2} - \frac{1 - {\beta_{n, k}}^{2}/2}{1 - \beta_{n, k}} + \frac{1 - \beta_{n, k}/2}{1 - \beta_{n, k}}\\
&= \frac{(1 - {\beta_{n, k}}^{2}) + {\beta_{n, k}}^{2} - \beta_{n, k}}{2(1 - \beta_{n, k})} = \frac{1}{2}.
\end{aligned}
\]
This completes the induction.
\end{proof}

\begin{proof}[Proof of Theorem~\ref{thm:SFTandTransitive}\ref{item:4:Density2}]
This is a direct consequence of Theorem~\ref{thm:Palmer}.
\end{proof}

For the proof of Theorem~\ref{thm:SFTandTransitive}\ref{item:3:Density2} we will require the following additional preliminaries.
\begin{enumerate}[label*=(\alph*)]
\item A \textit{Pisot number} is a positive real number $\beta$ whose Galois conjugates all have modulus strictly less than $1$.  A \textit{Perron number} is a positive real number $\beta$ whose Galois conjugates all have modulus strictly less than $\beta$.
\item For $n \geq 2$, the $n$-th multinacci number is a Pisot number (see \cite[Example 3.3.4]{DK:2002}).
\item If the space $\Omega_{\beta, 0}$ is sofic, then $\beta$ is a Perron number (see \cite[Theorem 2.2(3)]{S:2003}); here the result is attributed to \cite{B-M:1986,D:1984,P:1960}.
\end{enumerate}

\begin{proof}[Proof of Theorem~\ref{thm:SFTandTransitive}\ref{item:3:Density2}]
Since, by definition, any SFT is sofic, it is sufficient to show that $\beta_{n, k}$ is not a Perron number, for all $n, k \in \mathbb{N}$ with $n \geq 2$.  To this end let $P_{n, k}$ denote the polynomial given by
\begin{align*}
P_{n, k}(x) \coloneqq x^{2^{k} n} - x^{2^{k} (n-1)} - \dots - x^{2^{k}} - 1.
\end{align*}
Suppose, by way of contradiction, that $\beta_{n, k}$ was a Perron number.  By the definition of $\beta_{n, k}$ and $P_{n, k}$, we have that $P_{n, k}(\beta_{n, k}) = 0$, and so the minimal polynomial $Q_{n, k}$ of $\beta_{n, k}$ divides $P_{n, k}$.  Further, if $\beta$ was a root of $Q_{n, k}$ not equal to $\beta_{n, k}$, then it would also be a root of $P_{n, k}$.   Moreover, $\lvert \beta \rvert < 1$.  Conversely, suppose that $|\beta| > 1$. Since $\beta$ is a root of $Q_{n,k}$, then $P_{n,k}(\beta) = 0$.  Thus $1 <|\beta|< \lvert \beta^{2^{k}} \rvert <(\beta_{n,k})^{2^k} = \gamma_{n}$, by our assumptions.  Hence, $\beta^{2^k}$ would be a Galois conjugate of $\gamma_{n}$ with modules greater than $1$, contradicting the fact that $\gamma_{n}$ is a Pisot number.  Thus, all of the roots of $Q_{n, k}$ with the exception of $\beta_{n, k}$ would be of absolute value strictly less than $1$, and hence $\beta_{n, k}$ would be a Pisot number.  However, it is well known that there are only two Pisot numbers $\theta_{0}$ and $\theta_{1}$ less than $2^{1/2}$, and moreover, $\theta_{1} > \theta_{0} > 2^{1/3}$.  Thus, since $\gamma_{n} \in (1, 2)$, if $k \geq 2 $ we obtain a contradiction to our assumption that $\beta_{n, k}$ is a Perron number.   By numerical calculations we know that $1.8 > {\theta_{0}}^2$, $1.925 > {\theta_{1}}^2$, $P_{4,0}(1.8) < 0$ and $P_{4,0}(1.925) < 0$.  Therefore, for all $n \geq 4$, we have that $\theta_{0}$ nor $\theta_{1}$ is a root of $P_{n,1}$.  Furthermore,  for $n \in \{ 2, 3 \}$, it is a simple calculation to show that the minimal polynomials of $\theta_{0}$ and $\theta_{1}$, namely $x \mapsto x^3 - x - 1$ and $x \mapsto x^4 - x^3 - 1$ respectively, do not divide $P_{n, 1}$, for $n \in \{ 2, 3\}$.  Hence, $\theta_{0}, \theta_{1} \not\in \{ \beta_{2,1}, \beta_{3,1} \}$.  This yields a contradiction to the assumption that $\beta_{n, 1}$ is a Perron number.
\end{proof}

\section*{Acknowledgements}

\noindent This work was initiated at the \textit{Bremen Winter School on Multifractals and Number Theory} held in 2013.  We are also grateful to K.\ Dajani, T.\ Jolivet, T.\ Jordan, C.\ Kalle and \mbox{W.\ Steiner} for their comments during the writing of this article.  The third author also extends his thanks to M.\ Barnsley for many stimulating conversations concerning the extended model of an intermediate $\beta$-transformation.  Further, we would like to thank the referee for his/her detailed comments and suggestions.


\begin{thebibliography}{99}
%
\bibitem{AM:1996}
L. Alsed\'a and F. Manosas,
Kneading theory for a family of circle maps with one discontinuity,
\textit{Acta Math. Univ. Comenian.\ (N.S.)} \textbf{65} (1996), 11--22.
%
\bibitem{BM:2011}
M. F. Barnsley and N. Mihalache,
Symmetric itinerary sets,
\textit{preprint}, arXiv:{1110.2817v1}.
%
\bibitem{BV:2012}
M. Barnsley W. Steiner and A. Vince,
A combinatorial characterization of the critical itineraries of an overlapping dynamical system,
\textit{preprint}, arXiv:{1205.5902v3}.
%
\bibitem{BHV:2011}
M. F. Barnsley, B. Harding and A. Vince,
The entropy of a special overlapping dynamical system,
\textit{Ergodic Theory Dyn.\ Sys.\ - online} (2012).
%
\bibitem{B-M:1986}
A. Bertrand-Mathis,
D\'eveloppement en base $\theta$, r\'epartition modulo un de la suite $(x \theta^{n})_{n \geq 0}$; languages code\'es et $\theta$-shift,
\textit{Bull.\ Soc.\ Math.\ Fr.} \textbf{114} (1986), 271--323.
%
\bibitem{B:1989}
F. Blanchard,
$\beta$-expansions and symbolic dynamics,
\textit{Theoret.\ Comput.\ Sci.} \textbf{65} (1989), 131--141.
%
\bibitem{BS:2004}
M. Brin and G. Stuck,
\textit{Introduction to dynamical systems},
Cambridge University Press, 2004.
%
\bibitem{DK:2002}
K. Dajani and C. Kraaikamp,
\textit{Ergodic theory of numbers},
The Carus Mathematical Monographs 29, 2002.
%
\bibitem{DK:2002b}
K. Dajani and C. Kraaikamp,
From greedy to lazy expansions and their driving dynamics,
\textit {Expo. Math.} \textbf{20} (2002), 315-327.
%
\bibitem{DV:2005}
K. Dajani and M. deVries,
Measures of maximal entropy for random $\beta$-expansions,
\textit{J.\ Eur.\ Math.\ Soc.}  \textbf{7} (2005), 51--68.
%
\bibitem{DV:2007}
K. Dajani and M. deVries,
Invariant densities for random $\beta$-expansions,
\textit{J.\ Eur.\ Math.\ Soc.} \textbf{9} (2007), 157--176.
%
\bibitem{DDGV06}
I. Daubechies, R. DeVore, S. G\"unt\"urk and V. Vaishampayan,
A/D conversion with imperfect quantizers,
\textit{IEEE Trans.\ Inform.\ Theory} \textbf{52} (2006), 874--885.
%
\bibitem{EO:1994}
B. Eckhardt and G. Ott,
Periodic orbit analysis of the Lorenz attractor,
\textit{Zeit.\ Phys.\ B} \textbf{93} (1994), 258--266.
%
\bibitem{F:1990}
K. J. Falconer,
\textit{Fractal geometry: Mathematical foundations and applications},
2$^{nd}$ edition, John Wiley and Sons, 2003.
%
\bibitem{FW:2012} 
A.-H. Fan and B.-W. Wang, 
On the lengths of basic intervals in beta expansions,
\textit{Nonlinearity} \textbf{25} (2012), no.\ 5, 1329--1343.
%
\bibitem{FL:2009}
Ch.~Frougny and A.~C.~Lai,
On negative bases,
\textit{Proceedings of DLT 09, Lecture Notes in Comput. Sci., Springer, Berlin} \textbf{5583} (2009), 252--263.
%
\bibitem{G:1990}
P. Glendinning,
Topological conjugation of Lorenz maps by $\beta$-transformations,
\textit{Math.\ Proc.\ Camb.\ Phil.\ Soc.} \textbf{107} (1990), 401--413.
%
\bibitem{HMP:2013}
T.~Hejda, Z.~Mas{\'a}kov{\'a} and E.~Pelantov{\'a},
Greedy and lazy representations in negative base systems.
\textit{Kybernetika} \textbf{49} (2013), no.~2, 258--279.
%
\bibitem{H:1979}
F. Hofbauer,
Maximal measures for piecewise monotonically increasing transformations on $[0, 1]$,
\textit{Ergodic Theory Lecture Notes in Mathematics} \textbf{729} (1979), 66--77.
%
\bibitem{HS:1990}
J. H. Hubbard and C. T. Sparrow,
The classification of topologically expansive Lorenz maps,
\textit{Comm.\ Pure Appl.\ Math.} \textbf{XLIII} (1990), 431--443.
%
\bibitem{IS:2009}
S.~Ito and T.~Sadahiro,
Beta-Expansions with negative bases.
\textit{Integers} \textbf{9} (2009), no.\ A22, 239--259.
%
\bibitem{KS:2012}
C. Kalle and W. Steiner,
Beta-expansions, natural extensions and multiple tilings associated with Pisot units,
\textit{Trans.\ Amer.\ Math.\ Soc.} \textbf{364} (2012), 2281--2318.
%
\bibitem{KL:1998}
V. Komornik and P. Loreti,
Unique developments in non-integer bases,
\textit{Amer.\ Math.\ Monthly} \textbf{105} (1998), 636--639.
%
\bibitem{LS:2012}
L.~Liao and W.~Steiner, 
Dynamical properties of the negative beta-transformation
\textit{Ergodic Theory Dyn.\ Sys.} \textbf{32} (2012), 1673--1690.
%
\bibitem{D:1984}
D. Lind,
The entropies of topological Markov shifts and a related class of algebraic integers,
\textit{Ergodic Theory Dyn.\ Sys.} \textbf{4} (1984), 283--300.
%
\bibitem{LM:1995}
D. Lind and B. Marcus,
\textit{An introduction to symbolic dynamics and coding},
Cambridge University Press, 1995.
%
\bibitem{L:1963}
E. N. Lorenz,
Deterministic nonperiodic flow,
\textit{J.\ Atmos.\ Sci.} \textbf{20} (1963), 130--141.
%
\bibitem{P:1979}
M. R. Palmer,
\textit{On the classification of measure preserving transformations of Lebesgue spaces},
Ph.\ D.\ thesis, University of Warwick, 1979.
%
\bibitem{P:1960}
W. Parry,
On the $\beta$-expansions of real numbers,
\textit{Acta Math.\ Acad.\ Sci.\ Hungar.} \textbf{11} (1960), 401--416.
%
\bibitem{P1964}
W. Parry,
Representations for real numbers,
\textit{Acta Math.\ Acad.\ Sci.\ Hungar.} \textbf{15} (1964), 95--105.
%
\bibitem{P:1965}
W. Parry,
Symbolic dynamics and transformations of the unit interval,
\textit{Trans.\ Amer.\ Math.\ Soc.} \textbf{122} (1965), 368--378.
%
\bibitem{R:1957}
A. R\'enyi,
Representations for real numbers and their ergodic properties,
\textit{Acta Math.\ Acad.\ Sci.\ Hungar.} \textbf{8} (1957), 477--493.
%
\bibitem{S:2003}
N. Sidorov,
Arithmetic dynamics,
\textit{Topics in Dynamics and Ergodic Theory, LMS Lecture Notes Ser.} \textbf{310} (2003), 145--189.
%
\bibitem{S:2003b}
N. Sidorov,
Almost every number has a continuum of $\beta$-expansions,
\textit{Amer.\ Math.\ Monthly} \textbf{110} (2003), 838--842.
%
\bibitem{V:2003}
D. Viswanath,
Symbolic dynamics and periodic orbits of the Lorenz attractor,
\textit{Nonlinearity} \textbf{16} (2003), 1035--1056.
%
\bibitem{Will:1975}
K. M. Wilkinson,
Ergodic properties of a class of piecewise linear transformations,
\textit{Z. Wahrscheinlickeitstheorie verw. Gebiete} \textbf{31} (1975), 303--328.
%
\bibitem{W:1980}
R. F. Williams,
Structure of Lorenz attractors,
\textit{Publ.\ Math.\ IHES} \textbf{50} (1980), 73--100.
%
\end{thebibliography}
\end{document}